\documentclass[a4paper,oneside,11pt]{article}
\usepackage[utf8]{inputenc}
\usepackage{lmodern}
\usepackage[T1]{fontenc}
\usepackage{verbatim}
\usepackage{textcomp}
\usepackage[english]{babel}
\usepackage[a4paper,vmargin={3cm,3cm},hmargin={3cm,3cm}]{geometry}
\usepackage[font=sf, labelfont={sf,bf}, margin=1cm]{caption}
\usepackage[pdftex]{hyperref}
\usepackage[pdftex]{color,graphicx}
\usepackage{amssymb}
\usepackage{amsmath}

\renewcommand{\leq}{\leqslant}
\renewcommand{\geq}{\geqslant}

\usepackage{amsmath,amsfonts,amssymb,amsthm,mathrsfs}
\usepackage{stackrel}

\def\build#1_#2^#3{\mathrel{
\mathop{\kern 0pt#1}\limits_{#2}^{#3}}}

\theoremstyle{plain}
\newtheorem{theorem}{Theorem}
 
\newtheorem{corollary}{Corollary}

\newtheorem{lemma}[corollary]{Lemma}

\theoremstyle{definition}

\theoremstyle{remark}

\newcommand{\qua}{\mathbf{q}}
\newcommand{\map}{\mathbf{m}}
\newcommand{\Qua}{\mathcal{Q}}
\newcommand{\Map}{\mathcal{M}}
\newcommand{\ci}{\perp\!\!\!\perp}

\hypersetup{
    pdfauthor   = {Laurent M\'enard, Pierre Nolin},%
    pdftitle    = {Percolation on uniform infinite planar maps}}

\begin{document}
\title{Percolation on uniform infinite planar maps}
\author{Laurent M\'enard, Pierre Nolin}
\date{Universit\'e Paris Ouest, ETH Z\"urich}
\maketitle


\begin{abstract}
We construct the uniform infinite planar map (UIPM), obtained as the $n \to
\infty$ local limit of planar maps with $n$ edges, chosen uniformly at random.
We then describe how the UIPM can be sampled using a ``peeling'' process, in a
similar way as for uniform triangulations. This process allows us to prove that
for bond and site percolation on the UIPM, the percolation thresholds are
$p^{\textrm{bond}}_c=1/2$ and $p^{\textrm{site}}_c=2/3$ respectively. This method also works for other classes of random infinite planar maps, and we show in particular that for bond percolation on the uniform infinite planar quadrangulation, the percolation threshold is $p^{\textrm{bond}}_c=1/3$.
\end{abstract}

\section{Introduction}

\subsection{Background and motivations}

A lot of progress has been made in the past decade toward
the understanding of statistical physics models in dimension $2$. All these
models, when examined at their critical point, share a strong property of
conformal invariance, a property which has been established for a number of
them, in the scaling limit. Without aiming at exhaustivity, let us mention the
Loop-Erased Random Walk \cite{LSW_LERW}, site percolation on the triangular lattice \cite{Sm}, the
Ising model of ferromagnetism \cite{CS2} and its dual FK-Ising representation \cite{Sm3}.
This property leads to a precise description of geometric objects in terms of the Schramm-Loewner Evolution (SLE)
processes introduced in \cite{Sc}, and subsequently studied in a number of papers -- let us mention the groundbreaking works \cite{RS, LSW1, LSW2}, to name but a few.

For percolation in particular, this led to the derivation of the so-called ``arm
exponents'', that describe the probability of observing disjoint long-range
paths: for instance, at criticality, the probability for a given vertex to be
connected to distance $n$ follows a power law: it decays like
$n^{-\alpha'_1+o(1)}$ as $n \to \infty$, with $\alpha'_1=\frac{5}{48}$. Combining this new
understanding with Kesten's scaling relations \cite{Ke87}, one can then describe
the behavior of percolation not only at criticality, but also near criticality,
i.e. through its phase transition. Let us mention in particular that the density
of the infinite connected component decays as $\theta(p) = (p-p_c)^{\beta+o(1)}$
as $p \searrow p_c$, with $\beta = \frac{5}{36}$ \cite{SmW}.

Such exponents had however been predicted much earlier by powerful but non-rigorous
methods, such as quantum gravity. Let us mention in particular the paper \cite{AAD},
where arm exponents in their own right were first
considered and derived. Random graphs have been extensively used in the
statistical physics literature, with a view to analyzing random spatial
processes such as percolation or the Ising model. Studying these models in
random geometries can provide a useful insight on their
behavior on Euclidean lattices such as $\mathbb{Z}^2$ or the triangular lattice.
Once derived the critical exponents in the random
graph setting, the Knizhnik-Polyakov-Zamolodchikov (KPZ) formula \cite{KPZ} predicts what
the values of these exponents are for (regular) Euclidean lattices.

\bigskip

Let us now make a bit more precise what is meant by random geometries. In the following, we consider proper embeddings of finite connected graphs in the sphere $\mathbb{S}^2$, where loops and multiple edges are allowed. A \emph{finite
planar map} is then an equivalence class of such embeddings with respect to
orientation-preserving homeomorphisms of the sphere. A planar map is
\emph{rooted} if it has furthermore a distinguished oriented edge
$\vec{e}=(v_0,v_1)$, which is then called the \emph{root edge} ($v_0$ being the
{root vertex}). Faces of the map are the connected components of the
complement of the union of its edges, and a map is a \emph{$p$-angulation} if
all its faces have degree $p$. In particular, when $p=3$ (resp. $4$), we obtain
triangulations (resp. quadrangulations).

The set of vertices of a given map will always be equipped
with the graph distance. From this point of view, a random planar map can be
considered as a random discrete metric space, giving a precise mathematical
framework for two-dimensional quantum gravity. In particular, it
is believed that random planar maps provide a good approximation for
continuous random surfaces. Recently, Le Gall \cite{LG_bmap} and Miermont
\cite{Mie} proved that random planar $p$-angulations (for $p=3$ or $p \geq
4$ even) properly rescaled converge towards a universal random surface, called
the
Brownian Map, in analogy with the fact that the Brownian motion arises as the
scaling limit of discrete random walks.

In this paper, rather than dealing with continuous scaling limits, we consider
\emph{local limits} of random maps as introduced in \cite{BS}, which is a
natural way to construct random infinite planar graphs.
We define the distance $d$ as: for every pair of finite rooted maps $\mathbf{m},\mathbf{m}'$,
\[
d \left( \mathbf{m}, \mathbf{m}' \right) =  \left( 1 + \sup \left\{ r: \, B_{r}(
\mathbf{m}) = B_{r}(\mathbf{m}') \right\} \right)^{-1}
\]
where, for $r \geqslant 1$, $B_{r}(\mathbf{m})$ is the planar map consisting of all edges of $\mathbf{m}$ that have at least one vertex at distance
strictly smaller than $r$ from the root (and $\sup \emptyset = 0$ by
convention).
We denote by $(\Map,d)$ the completion of the space of all finite rooted maps with
respect to $d$. Elements of $\Map$ that
are not finite maps are called \emph{infinite maps}. Note that one can extend
the function defined for finite maps $\mathbf{m} \mapsto
B_{r}(\mathbf{m})$ to a continuous function $B_{r}$ on $\Map$. The ball
$B_{r}(\mathbf{m})$ can be interpreted in a natural way as the union of the
edges of
$\mathbf{m}$ that have a vertex at distance strictly smaller than $r$ from the
root.

In a pioneering work \cite{AS}, Angel and Schramm constructed the uniform
infinite planar triangulation (UIPT) as the local limit of uniformly distributed large
triangulations. Shortly after, Krikun \cite{Kr} defined
similarily the uniform infinite
planar quadrangulation (UIPQ): if $\qua_n$ is distributed
according to the uniform measure on the set of all rooted quadrangulations with
$n$ faces, then it is proved in \cite{Kr} that the
distribution of $\qua_n$ converges weakly to a probability measure $\tau$ in the
set of all probability measures on infinite quadrangulations: the measure $\tau$
is the law of the UIPQ. Both the UIPT
and the UIPQ have been the focus of numerous works in recent years such as
\cite{An1,CD,CMM,Kri04,LGM,M}, but it is fair to say that they are not yet
fully understood. In this paper, we study in more detail independent percolation on these objects.

\subsection{Organization of the paper and main results}

In Section \ref{main tools}, we remind several important properties of planar
quadrangulations that will be instrumental in the present paper. We also
describe a natural bijection between quadrangulations and planar maps. This
bijection allows one to use properties for quadrangulations in order to study planar maps. In
particular, it provides an easy way to construct the uniform infinite planar map
(UIPM) from the uniform infinite planar quadrangulation (UIPQ). This bijection
also behaves nicely through restrictions. In particular, uniform infinite planar
$p$-angulations could also be constructed in this way.

\bigskip

We then describe in Section \ref{peeling} a ``peeling process'' similar to the process introduced by Angel in \cite{An1} for triangulations. This process offers a useful description of the usual exploration process, that follows the interface between black (occupied) and white (vacant) sites, as a simple Markov chain for which the transition probabilities are known rather explicitly.

\bigskip

In \cite{An1}, Angel used this description to study site percolation on the
uniform infinite planar triangulation (UIPT). The usual planar triangular
lattice has a ``self-matching'' property that suggests that for site percolation on
this lattice, one has $p_c=1/2$, which is a celebrated result of Kesten
\cite{Ke0}. The UIPT is ``stochastically'' self-matching, and it
also holds in this case that $p_c=1/2$, in both annealed and quenched
environments. Similarly, $\mathbb{Z}^2$ has a self-duality property that implies
that $p_c=1/2$ for bond percolation on this lattice (strictly speaking, this is
the actual result proved in \cite{Ke0}). The UIPM happens to be
``stochastically'' self-dual too, and in Section \ref{bond_percolation} we use
the peeling process to prove that $p_c^{\textrm{bond}}=1/2$ a.s. in this case. Before
that, we derive the site percolation threshold on the UIPM in Section \ref{site_percolation}, where we show that
$p_c^{\textrm{site}}=2/3$ a.s. In the last part of Section \ref{bond_percolation}, we explain how the method allows one to compute bond percolation thresholds for other classes of random infinite planar maps. In particular, we show that for bond percolation on the UIPQ, $p_c^{\textrm{bond}} = 1/3$ a.s. The main result of our paper is thus 
the following.

\begin{theorem} \label{main_theorem}
For site and bond percolation on the UIPM, one has, respectively, $p_c^{\textrm{site}}=2/3$ and
$p_c^{\textrm{bond}}=1/2$ almost surely. For bond percolation on the UIPQ, one has $p_c^{\textrm{bond}} = 1/3$ almost surely.
\end{theorem}

%

\section{Main tools} \label{main tools}

\subsection{Quadrangulations and planar maps}

Recall that a finite planar map is a \emph{quadrangulation} if all its faces
have degree
$4$, that is $4$ adjacent edges. Note that the underlying graph of a
quadrangulation is bipartite. A planar map is a \emph{quadrangulation with a
boundary} or with \emph{holes} if all its faces have degree $4$, except for a
number of distinguished
faces which can be arbitrary even-sided polygons (we assume these boundaries
to be simple, i.e. the polygons are not ``folded''). In the case when there is only one hole, of
perimeter $2p$, we obtain what is called a quadrangulation of the
$2p$-gon. For every integer $n \geqslant
0$, we denote by $\Qua_n$ the set of all rooted quadrangulations with $n$
faces. We also denote by $\Qua_{n}^p$ the set of all quadrangulations of the
$2p$-gon with $n$ inner faces, such that the external face contains the root edge and lies on the right-hand side of it. The completion of
the set $\Qua_f$ of all
rooted finite
quadrangulations for the distance $d$ (defined in the introduction) is denoted by $\Qua$; it is a
subset of $\Map$. Elements of $\Qua_{\infty} = \Qua
\setminus \Qua_f$ are called infinite rooted quadrangulations. Similarily, we
denote the set of finite (resp. infinite) quadrangulations of the $2p$-gon by
$\Qua_{f}^p$ (resp. $\Qua_{\infty}^p$). We refer to \cite{CMM} for more details.

For our purpose, when dealing with quadrangulations, it turns out to be more convenient to
work with faces rather than with edges, which leads us to introduce the new distance
$d^{\star}$ on $\Map$ defined by
\[
d^{\star} \left( \mathbf{m}, \mathbf{m}' \right) =  \left( 1 + \sup \left\{ r:
\, B^{\star}_{r}( \mathbf{m}) = B^{\star}_{r}(\mathbf{m}') \right\} \right)^{-1}
\]
for all rooted maps $\mathbf{m}$, $\mathbf{m}'$, where, for $r \geqslant
1$, we denote by $B^{\star}_{r}(\mathbf{m})$ the planar map obtained as the union of all
faces of $\mathbf{m}$ that have at least one vertex at distance strictly smaller
than $r$ from the root. Note that if $\mathbf{m}$ is a
quadrangulation, then $B^{\star}_{r}(\mathbf{m})$ is a quadrangulation with
holes.

\bigskip

Let us stress that if maps with faces of
arbitrarily large degrees are considered, then the distances
$d$ and $d^{\star}$ give rise to two different topologies. However, when restricted to quadrangulations, the two distances are
equivalent (more generally, this holds true for closed sets of maps with faces
of bounded degree). Indeed, for every $\mathbf{q} \in \Qua$ and $r \geq 1$, one
has
$$B_r(\mathbf{q}) \subset B^{\star}_r(\mathbf{q}) \subset B_{r+2}(\mathbf{q}).$$
Therefore, $\Qua$ can also be seen as the completion of $\Qua_f$ for the distance $d^{\star}$.

\bigskip

If we are given a quadrangulation with holes, it is natural to construct a full
quadrangulation of the sphere by filling its holes of degree $2p$ with
quadrangulations of the $2p$-gon. However, one has to make sure that filling
the holes with different quadrangulations leads to different maps. This can be
ensured by dealing only with \emph{rigid} quadrangulations, as in \cite{AS}: we say that a
rooted quadrangulation with holes $\qua$ is rigid if no quadrangulation
of the sphere  includes two different copies of $\qua$ with coinciding roots. As
stated in \cite{BC}, an easy adaptation of Lemma 4.8 of \cite{AS}
yields that any rooted quadrangulation with holes is rigid.

\subsection{UIPQ and UIPM}
\label{bijection}

As we already mentioned, the law of the UIPQ can be constructed as the weak local limit of
uniform measures on large quadrangulations.
Recently, Curien and Miermont \cite{CM} constructed similar measures for quadrangulations with a boundary. More precicely, if
$\qua^p_n$ is distributed according to the uniform measure on $\Qua_n^p$,
then the distribution of $\qua^p_n$ converges
weakly to a probability measure $\tau^p$ in the set of all probability measures
on $\left( \Qua^p, d^{\star} \right)$: the measure $\tau^p$ is the law of the
uniform
infinite planar quadrangulation of the $2p$-gon.

\bigskip

There is a natural bijection between rooted quadrangulations and rooted planar maps, which we now describe (see Figure \ref{bijection_fig}).
Starting from a quadrangulation, its bipartite structure allows one to divide its
set of vertices into two sets: \emph{circle-vertices} are the vertices which
are at an even distance from the root vertex (including the
root vertex itself), and \emph{square-vertices} are the vertices at an odd
distance. Now, draw an edge between any two
circle-vertices on the same face: we produce in this way a planar map with $n$
edges, rooted at the edge corresponding to the face (in the initial quadrangulation) which is on the left hand-side of the root
edge. Making explicit the reverse map is
straightforward: it suffices to add one square-vertex on each face, and connect
it to all vertices of this face. This bijection is used throughout the paper.

\begin{figure}[ht!]
\begin{center}
\includegraphics[width=13cm]{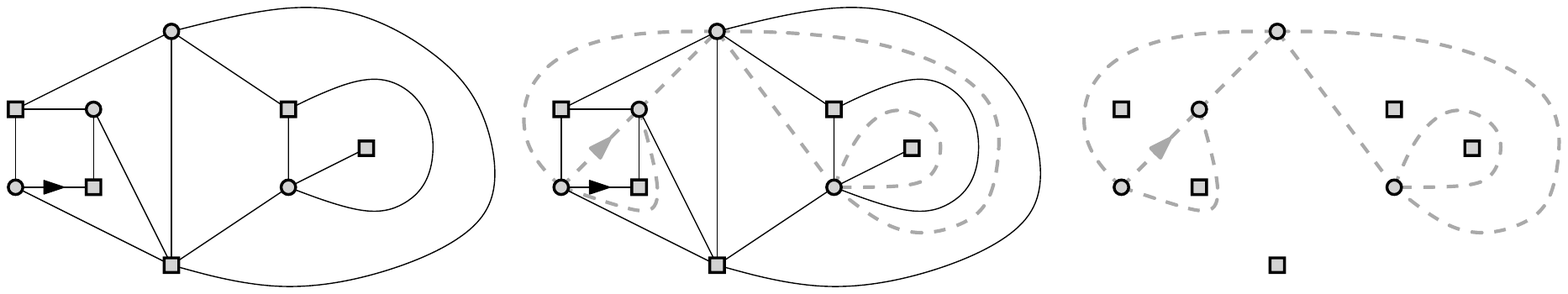}
\caption{\label{bijection_fig}The bijection between quadrangulations and planar
maps.}
\end{center}
\end{figure}

This bijection maps the uniform measure on rooted quadrangulations with $n$
faces to the uniform measure on rooted planar maps with $n$ edges. Therefore,
it can be used to define a (random) uniform infinite planar map (UIPM), whose law is
just the weak limit of the uniform measure on rooted planar maps with $n$ edges for
the distance $d$. Indeed, it is easy to check that this bijection is
continuous for the topologies considered.

Note also that this bijection maps the circle-vertices onto the vertices of the
final map, while the square-vertices are mapped to faces. The dual graph of
the random map can thus be obtained by simply choosing to draw edges between
square-vertices, instead of between circle-vertices. This also corresponds
to re-rooting the original quadrangulation by reversing orientation of the root
edge.

The planar map so obtained is thus ``stochastically'' self-dual (because the
uniform infinite quadrangulation is invariant under the previous re-rooting,
or because the dual of a random uniform planar map with $n$ edges has the same
law), which seems to indicate that the bond percolation threshold on this map is $p_c=1/2$, as in the case of $\mathbb{Z}^2$ (see \cite{Ke0}) which is ``truly'' self-dual.

\subsection{Counting quadrangulations}
\label{combinatorics}

In this short section, we collect some enumeration results for quadrangulations that are
instrumental for our purpose. We refer the reader to \cite{BG} for proofs.

If we denote by $a_{n,p}$ the number of quadrangulations of the $2p$-gon with $n$
internal faces rooted on the boundary face, one has
\begin{equation} \label{counting_quad}
a_{n,p} = 3^n \frac{(2p)!}{p! (p-1)!} \frac{(2n+p-1)!}{n!
(n+p+1)!}.
\end{equation}
Actually, the exact value of $a_{n,p}$ is not needed, but only its asymptotic behavior:
\begin{align}
a_{n,p} & \underset{n \to \infty}{\sim} C_p 12^n n^{-5/2}, \label{asymptotic_behavior1}\\
\text{ with } C_p & = \frac{1}{2 \sqrt{\pi}}\left( \frac{2}{3}\right)^p
\frac{(3p)!}{p!(2p-1)!}.
\end{align}
We also need asymptotic expressions for the corresponding generating functions:
$$Z_p(t) := \sum_{n \geqslant 0} a_{n,p} t^n$$
has $1/12$ as a convergence radius, and
\begin{equation} \label{Z_p}
Z_p := Z_p(1/12) = 2 \left(\frac{2}{3} \right)^p \frac{(3p-3)!}{p! (2p-1)!}.
\end{equation}

Following \cite{AS,BC}, we define the free distribution on rooted
quadrangulations of a $2p$-gon as the probability measure $\mu^p$
that assigns the weight
\[
\mu^p(\qua) = \frac{12^{-n}}{Z_p(1/12)}
\]
to each quadrangulation $\qua$ of the $2p$-gon having $n$ internal
faces and rooted on its boundary face.

\subsection{Spatial Markov property for the UIPQ}

We now state the spatial Markov property of the UIPQ, grouping into a unique lemma all the properties that are needed. 

\begin{lemma}
Let us denote by $\qua_{\infty}$ the UIPQ, and let $\mathbf{q}$ be a rigid quadrangulation with
$n$ internal faces and $k$ boundary faces, with perimeters $2p_1, \ldots , 2p_k$.

\begin{itemize}
\item[(i)] One has
\begin{equation}
\tau \left( \mathbf{q} \subset \mathbf{q}_{\infty} \right) = \frac{12^{-n}}{C_1}
\left(\prod_{i = 1}^k Z_{p_i} \right) \sum_{i=1}^{k}
\frac{C_{p_i}}{Z_{p_i}}.
\end{equation}

\noindent When $\mathbf{q} \subset \mathbf{q}_{\infty}$ holds, let us denote by $\mathbf{q}_i$ the component of the UIPQ in the $i$'th face.

\item[(ii)] Almost surely, only one of these components is infinite: the
probability that it is $\mathbf{q}_j$ is given by the $j$'th term in the
previous sum, i.e.
\begin{equation} \label{spatial_Markov2}
\tau \left( \mathbf{q} \subset \mathbf{q}_{\infty}, \text{$\mathbf{q}_j$ is infinite} \right) = \frac{12^{-n}}{C_1} C_{p_j} \left(\prod_{\substack{i=1\\ i \neq j}}^k Z_{p_i}
\right).
\end{equation}

\item[(iii)] If we condition on the event that $\left\{ \mathbf{q} \subset \mathbf{q}_{\infty}
\right\}$, and that the external faces of $\mathbf{q}$ all contain finitely many vertices of $\mathbf{q_\infty}$, except (possibly) the $j$'th
one, then
\begin{itemize}
 \item the quadrangulations $(\mathbf{q}_i)_{1 \leq i \leq k}$ are independent,
 \item $\mathbf{q}_j$ has the same distribution as the UIPQ of the $2 p_j$-gon,
 \item and for $i \neq j$, $\mathbf{q}_i$ is distributed as the free
quadrangulation
of a $2 p_i$-gon.
\end{itemize}

\end{itemize}

\end{lemma}

This spatial Markov property is proved in \cite{AS} for uniform triangulations, and a strictly identical proof applies in our
setting of quadrangulations.

\section{Peeling process for quadrangulations} \label{peeling}

We now describe the peeling process, a growth process that can be used to sample planar maps. It has first been
used in physics \cite{ADJ} to derive heuristics for the scaling limit of
2-dimensional quantum gravity. Later, Angel \cite{An1} defined rigourously this
process for triangulations, and used it to study volume growth and site
percolation on the UIPT. Benjamini and Curien \cite{BC} adapted this process to
quadrangulations in order to prove that the simple random walk on the UIPQ is
subdiffusive. We will make extensive use of this process to study both site
and bond percolation on the UIPM associated with the UIPQ.

\bigskip

Let $\qua_{\infty}$ be the UIPQ. The peeling process is a sequence $\left(
\qua_n \right)_{n \geqslant 0}$ of (finite) random quadrangulations with simple
boundary, such that:
\begin{itemize}
 \item $\qua_0$ is the root edge of $\qua_{\infty}$ and one has $\qua_0 \subset
\qua_1 \subset \cdots \subset \qua_n \subset \cdots \subset \qua_{\infty}$.
 \item Let $\mathcal{F}_n$ be the fitration generated by $\qua_0 ,
\qua_1 , \ldots, \qua_n$. Then conditionally on $\mathcal{F}_n$, the part of
$\qua_{\infty}$ that has not been discovered yet, that is $\qua_{\infty} \setminus
\qua_n$, is a UIPQ of the $\left|\partial \qua_n \right|$-gon.
\end{itemize}

Let us now describe the conditional distribution of $\qua_{n+1}$ knowing
$\mathcal{F}_n$, and write down explicit transition probabilities. First, we have to
choose an oriented edge $e$ on $\partial \qua_n$. Any choice, deterministic or
random, is acceptable as long as it depends only on $\mathcal{F}_n$, and
$\qua_n$ lies on the right hand side of $e$. The map
$\qua_{\infty} \setminus
\qua_n$ rooted at $e$ is a UIPQ of the $\left|\partial \qua_n \right|$-gon. Let
$p = \left| \partial \qua_n \right|/2$, and denote the
vertices of $\partial \qua_n$ by $x_1, \ldots, x_{2p}$ so that $e=
(x_{2p},x_1)$ (see Figure \ref{peeling_fig}).
Now, let us
reveal the face of $\qua_{\infty} \setminus
\qua_n$ containing $e$. Following the orientation given by $e$, we denote the
vertices of this face by $(x_{2p}, x_1, y_0, y_1)$. Four cases may occur,
depending on whether $y_0$ and / or $y_1$ belong to $\partial \qua_n$: we now describe
$\qua_{n+1}$ in each case, and give the corresponding probability.

\begin{figure}[ht!]
\begin{center}
\includegraphics[width=13cm]{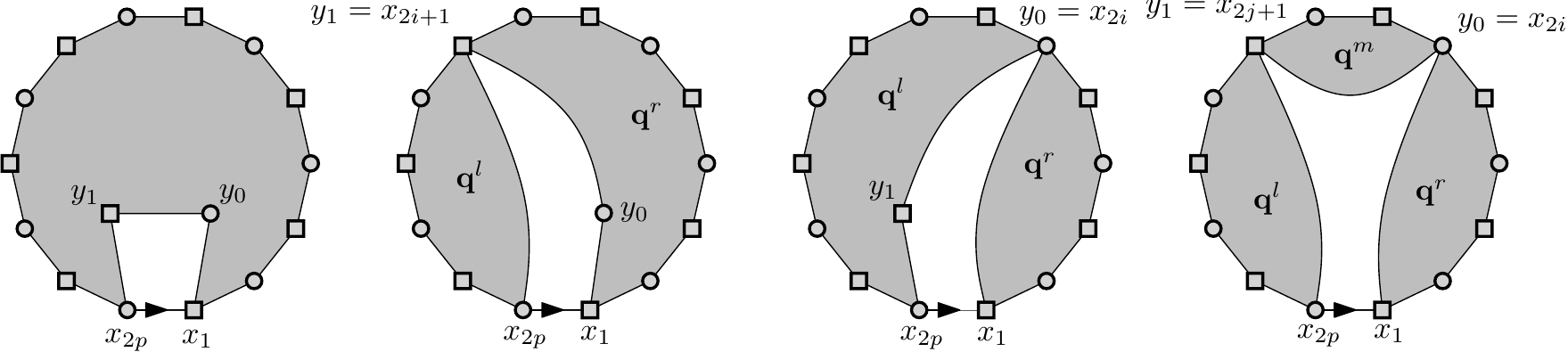}
\caption{\label{peeling_fig}Discovering a new face during the peeling process.
Note that $y_1$ can coincide with $x_1$,
and $y_0$ can coincide with $x_{2p}$ -- in the second and third cases,
respectively.}
\end{center}
\end{figure}

\begin{itemize}
 
\item[(1)] $y_0, y_1 \notin{\partial \qua_n}$ (Figure \ref{peeling_fig}, left). In this case, we set $\qua_{n+1}$ to be the union of $\qua_n$ and the face discovered. Therefore, $\qua_{\infty} \setminus \qua_{n+1}$ is a
quadrangulation of a $2(p+1)$-gon, and the spatial Markov property
ensures that conditionally on this event and $\mathcal{F}_n$, the map
$\qua_{\infty} \setminus \qua_{n+1}$ is a UIPQ of the $2(p+1)$-gon. Hence, conditionally on $\mathcal{F}_n$, this event has probability
\begin{equation*}
\tau \left( y_0, y_1 \notin \partial \qua_n \middle| \mathcal{F}_n \right) =
\tau^p \left( y_0, y_1 \notin \partial \qua^p  \right) =
\lim_{N\to \infty} \frac{a_{N-1,p+1}}{a_{N,p}}= \frac{C_{p+1}}{12 C_p}
\end{equation*}
(using \eqref{asymptotic_behavior1}).

\item[(2)] $y_0 \notin \partial
\qua_n$ and $y_1 = x_{2i +1}$ with $0 \leqslant i \leqslant p-1$ (Figure \ref{peeling_fig}, middle left). 
In this case, the new face divides the remaining part of $\qua_{\infty}$ into two
separate quadrangulations: $\qua_n^r$ with perimeter $2(i + 1)$ and $\qua_n^l$
with perimeter $2(p - i)$. Conditionally on this event and $\mathcal{F}_n$, exactly
one of these two quadrangulations is infinite. If it is $\qua_n^r$, the spatial
Markov property ensures that it is a UIPQ of the $2 (i+1)$-gon,
while $\qua_n^l$ is independent of $\qua_n^r$ and is a free
quadrangulation of the $2 (p-i)$-gon. We set $\qua_{n+1}$ to be the union of
$\qua_n$, the face discovered, and $\qua_n^l$, so that $\qua_{\infty} \setminus
\qua_{n+1} = \qua_n^r$. Using \eqref{spatial_Markov2}, the probability of this event is given by
\begin{equation*}
\tau \left( y_0 \notin \partial \qua_n, y_1  = x_{2i+1}, \qua_n^r \text{
infinite } \middle| \mathcal{F}_n \right) = \frac{Z_{p-i}C_{i+1}}{12 C_p}.
\end{equation*}
If $\qua_n^l$ is infinite, the situation is similar, and we set $\qua_{n+1}$ to be the union of
$\qua_n$, the face discovered, and $\qua_n^r$, so that $\qua_{\infty} \setminus
\qua_{n+1} = \qua_n^l$. The corresponding probability is
\begin{equation*}
\tau \left( y_0 \notin \partial \qua_n, y_1 = x_{2i+1}, \qua_n^l \text{
infinite } \middle| \mathcal{F}_n \right) = \frac{C_{p-i}Z_{i+1}}{12 C_p}.
\end{equation*}

\item[(3)] $y_1 \notin \partial \qua_n$ and $y_0 = x_{2i}$ with $1 \leqslant i \leqslant p$ (Figure \ref{peeling_fig}, middle right). The situation is similar to the second case, and conditionally on this event and $\mathcal{F}_n$, either
$\qua_n^r = \qua_{\infty} \setminus \qua_{n+1}$ is a UIPQ of the $2i$-gon, or
$\qua_n^l = \qua_{\infty} \setminus \qua_{n+1}$ is a UIPQ of the
$2(p+1-i)$-gon. The respective probabilities are:
\begin{align*}
\tau \left( y_1 \notin \partial \qua_n, y_0 = x_{2i}, \qua_n^r \text{
infinite } \middle| \mathcal{F}_n \right) 
& = \frac{Z_{p+1-i}C_{i}}{12 C_p}, \\
\tau \left( y_1 \notin \partial \qua_n, y_0 = x_{2i}, \qua_n^l \text{
infinite } \middle| \mathcal{F}_n \right) 
& = \frac{C_{p+1-i}Z_{i}}{12 C_p}.
\end{align*}

\item[(4)] $y_0 = x_{2i}$ and $y_1 =
x_{2j+1}$ with $1 \leqslant i \leqslant j \leqslant p-1$ (Figure \ref{peeling_fig}, right). In this case, the new face divides
the remaining part of $\qua_{\infty}$ into three
separate quadrangulations: $\qua_n^r$ with perimeter $2i$, $\qua_n^m$ with perimeter $2(j-i +1)$, and $\qua_n^l$
with perimeter $2(p - j)$.
Here again, the spatial Markov property ensures that conditionally on the corresponding event and
$\mathcal{F}_n$, exactly one of these quadrangulations is infinite, and we set
$\qua_{n+1}$ to be the union of $\qua_n$, the face discovered, and the other two finite
quadrangulations. The corresponding probabilities are
given by:
\begin{align*}
\tau \left( y_0 = x_{2i} , y_1 = x_{2j+1}, \qua_n^r \text{
infinite } \middle| \mathcal{F}_n \right) 
& = \frac{Z_{p-j}Z_{j-i+1}C_{i}}{12 C_p},\\
\tau \left( y_0 = x_{2i} , y_1 = x_{2j+1}, \qua_n^m \text{
infinite } \middle| \mathcal{F}_n \right) 
& = \frac{Z_{p-j}C_{j-i+1}Z_{i}}{12 C_p},\\
\tau \left( y_0 = x_{2i} , y_1 = x_{2j+1}, \qua_n^l \text{
infinite } \middle| \mathcal{F}_n \right) 
& = \frac{C_{p-j}Z_{j-i+1}Z_{i}}{12 C_p}.
\end{align*}
\end{itemize}

\bigskip

Let us insist on the fact that the peeling procedure that we just described, and its
transition probabilities, do not depend on the choice of the edge $e$, provided that at each
step $n$, this choice depends only on $\mathcal{F}_n$. This will allow us to study both site and bond percolation on the UIPM by
following the percolation interface along the way. This peeling procedure also allowed
Benjamini and Curien \cite{BC} to study the simple random walk on the UIPQ.

A more straightforward yet very useful consequence of this fact is that the
sequence $\left( |\partial \qua_n | , | \qua_n | \right)_{n \geqslant 0}$ is a
homogeneous Markov chain whose transition probabilities do not depend on the
particular peeling process performed. For instance, let us write $|\partial \qua_{n + 1}
| = | \partial \qua_n | + 2 X_n$ for every $n \geqslant 0$. Then one has, for this increment $X_n$, using the transition probabilities of the peeling process that we derived explicitly,
\begin{equation} \label{stepsize1}
P \left( X_n = 1 \big| |\partial \qua_n | = 2p \right) = \frac{C_{p+1}}{12 C_p}
\end{equation}
(corresponding to case (1) above), and for every $k = 0, \ldots, p-1$,
\begin{equation} \label{stepsize2}
P \left( X_n = -k \big| |\partial \qua_n | = 2p \right) = 4 \frac{C_{p- k} Z_{k+1}}{12 C_p} + 3 \frac{C_{p- k}}{12 C_p} \sum_{i = 1}^k
Z_i Z_{k+1 -i}
\end{equation}
(combining cases (2) and (3) for the first term, and (4) for the second term). Of particular
interest is the following asymptotics proven in Theorem 5 of \cite{BC}:
\begin{lemma}
\label{peelingasymptotics}
If $\qua_0, \qua_1 , \ldots , \qua_n, \ldots$ is generated by a peeling procedure of the
UIPQ, then one has
\begin{align*}
|\partial \qua_n | & \approx n^{2/3}, \\
|\qua_n| & \approx n^{4/3},
\end{align*}
where, if $(Y_n)_{n \geqslant 0}$ is a random
process, $Y_n \approx n^{\alpha}$ means that for some $\kappa > 0$, $Y_n n^{-\alpha} \log^{\kappa} (n) \to \infty$ and $Y_n
n^{-\alpha} \log^{- \kappa} (n) \to 0$ almost surely.
\end{lemma}
This property is proved in \cite{BC} by using geometric properties of the UIPQ, without appealing to the peeling process directly. However, it should also be possible to prove these
asymptotics by using the explicit transition probabilities for the peeling
process, and the enumeration results of Section \ref{combinatorics}. An easy
consequence of Lemma \ref{peelingasymptotics} -- actually, only the fact that $|\partial \qua_n | \to \infty$ a.s. -- that will be useful for our
purpose is the following:
\begin{corollary}
\label{expectX}
Let $\qua_0, \qua_1 , \ldots , \qua_n, \ldots$ be generated by a peeling procedure of the
UIPQ, and set $|\partial \qua_{n + 1}
| = | \partial \qua_n | + 2 X_n$ for every $n \geqslant 0$. Then one has
\begin{equation*}
 E \left[ X_n \middle| \mathcal{F}_n \right] \underset{n \to \infty}{
\longrightarrow} 0.
\end{equation*}
\end{corollary}

\begin{proof}
For $p > 0$ and $1 \leqslant k \leqslant p-1$, one can easily derive from
\eqref{stepsize2}:
\begin{equation*}
P \left( X_n = -k \big| |\partial \qua_n | = 2p \right) =
\frac{\left(p - \frac{1}{2} \right)_k \left(p - 1 \right)_k}{\left(p -
\frac{1}{3} \right)_k \left(p - \frac{2}{3} \right)_k } \, q_k
\end{equation*}
where $(x)_k = x (x-1) \cdots (x-k + 1)$ and
\begin{equation} \label{q_k}
q_k = \frac{1}{3}
Z_{k+1} \left( \frac{2}{9} \right)^k  + \frac{1}{4}
\sum_{i = 1}^k
Z_i Z_{k+1 -i} \left( \frac{2}{9} \right)^k.
\end{equation}
Therefore, the probabilities $P \left( X_n = -k \big| |\partial \qua_n | = 2p
\right)$ are increasing in $p$ and converge to $q_k$. Let us denote by $X$ a
random variable with law given by 
\begin{align*}
P(X = -k) & = q_k, \, \text{for $k \geqslant 1$},\\
P(X= 0) & = P \left( X_n = 0 \big| |\partial \qua_n | = 2p
\right)= \frac{1}{3} Z_1 = 4/9, \\
P(X=1) & = \lim_{p \to
\infty} P \left( X_n = 1 \big| |\partial \qua_n | = 2p
\right) = 3/8.
\end{align*}
Since $|\partial \qua_n| \to \infty$ almost
surely as $n$ grows, an argument of dominated convergence shows that $E
\left[ X_n \middle| \mathcal{F}_n \right]$ converges to $E[X]$.

Now, let us show that $E[X] = 0$. To this aim, we introduce the series
\begin{equation*}
 Z(x) = \sum_{k \geqslant 1} Z_k x^k,
\end{equation*}
with convergence radius $2/9$. The series $Z'(x)$ corresponds to the generating
series of ternary trees, and classical arguments (see \cite{BG}, (5.29)) yield
\begin{align*}
Z(x) & = - \frac{2}{3} + \frac{2}{3} \, _2F_1 \left(-\frac{2}{3} , - \frac{1}{3}
; \frac{1}{2} ; \frac{9x}{2} \right), \\
\text{and} \quad Z'(x) & = 4 \sqrt{\frac{2}{9x}} \sin \left( \frac{1}{3} \arcsin \left(
\sqrt{\frac{9x}{2}} \right) \right).
\end{align*}
From here, basic computations give
\begin{equation*}
\sum_{k \geqslant 1} k q_k  = \frac{1}{3} \left( Z' \left( \frac{2}{9} \right)
- Z_1 \right) + \frac{1}{4} \left( Z^2 \right)' \left( \frac{2}{9} \right) -
\sum_{k \geqslant 1} q_k = \frac{3}{8}.
\end{equation*}
\end{proof}

\bigskip

To conclude this section, let us stress that the peeling procedure for
the UIPQ also provides a sampling of the UIPM. Indeed, consider $\left(
\qua_n \right)_{n \geqslant 0}$ a peeling-generated sequence for the
UIPQ, and, for every $n \geqslant 0$, let $\map_n$ be the map associated with
$\qua_n$ by the bijection of Section \ref{bijection}: there is an edge of
$\map_n$ inside each face of $\qua_n$ (except for the boundary face). We obtain in this way an
increasing sequence of maps, which are all submaps of $\map_{\infty}$.

Note that different quadrangulations $\qua_n$ may produce the same map $\map_n$. In
fact there is more information on the UIPM in $\qua_n$ than in $\map_n$, since
$\qua_n$ also gives information on the faces of $\map_{\infty}$. Indeed, let us
consider two edges of $\map_n$. Considering only $\map_n$, it is not possible
to say if the two edges are part of the same face in $\map_{\infty}$. However,
this information is
available in $\qua_n$: the two edges belong to the same face of $\map_{\infty}$
\emph{iff} their associated quadrangles share a common square-vertex in $\qua_n$. This is not
problematic for our purpose, since we are not interested in the sequence
$(\map_n)_{n \geqslant 0}$ by itself.

\section{Site percolation on the UIPM} \label{site_percolation}

In this section, we consider Bernoulli site percolation on the UIPM: the vertices are colored, independently of each other, \emph{black} with probability $q$, and \emph{white} with probability $(1-q)$. We prove the first part of Theorem \ref{main_theorem}: for site percolation on the UIPM, the percolation threshold is almost surely
$$p_c^{\textrm{site}}=2/3.$$

\subsection{Exploration process}
\label{expsite}

Consider $\map_{\infty}$ the UIPM, and $\qua_{\infty}$ the associated UIPQ.
Suppose that each vertex of $\map_{\infty}$ is colored independently at random,
black with probability $q$ and white with probability $(1-q)$ (in
$\qua_{\infty}$, this corresponds to a coloring of circle-vertices only). We are interested in percolation of the origin, i.e. the existence of an infinite black connected component containing the origin.

We also assume for simplicity that the root vertex of $\map_{\infty}$ -- which is also the root vertex
of $\qua_{\infty}$ -- is colored black. We can sample percolation on the UIPM
simultaneously with a peeling process of the UIPQ: each time a new vertex of
the UIPM is added, we color it randomly, independently of all previous steps.
Note that if at some step $n$, all the vertices of the UIPM that are on the
boundary $\partial \qua_n$ are white, then these vertices separate from infinity
(in $\map_{\infty}$) the root vertex, which therefore does not percolate (for black sites).

Now, recall that we can choose where the next quadrangle is revealed at each step of the
peeling process. In particular, we can let this choice depend on the percolation
configuration sampled so far. On the one hand, if all the vertices of the UIPM that are on the
boundary $\partial \qua_n$ have the same color, then we can make an arbitrary choice.
On the other hand, if there are white and black vertices on $\partial \qua_n$, then we can ensure that $\partial \qua_n$ remains
divided in two arcs: one arc with black vertices only, and the other one with white
vertices only. If we then follow the orientation of the boundary, there is a
unique choice of three consecutive vertices $x_{2p}$, $x_1$, and $x_2$, where
$x_{2p}$ and $x_2$ are black and white respectively, and $x_1$ is a
square-vertex between them.  We then reveal
the quadrangle on the left side of the oriented edge $(x_{2p},x_1)$ (see Figure \ref{siteperc}).

If this rule is followed, it is easy to see that all black vertices on $\partial \qua_n$
belong to the percolation cluster containing the root vertex of
$\map_{\infty}$, as long as the boundary does not become totally white, which corresponds to detecting a white circuit. However, note that white vertices of $\partial \qua_n$ do not
necessarily belong to the same white cluster, so black and
white sites do not play symmetric roles in this process: one cannot simply use the symmetry $q \leftrightarrow 1-q$. The connectedness of white sites corresponds to ``$*$-connectedness'', as it is usually called for percolation on planar graphs such as $\mathbb{Z}^2$.

\begin{figure}[ht!]
\begin{center}
\includegraphics[width=13cm]{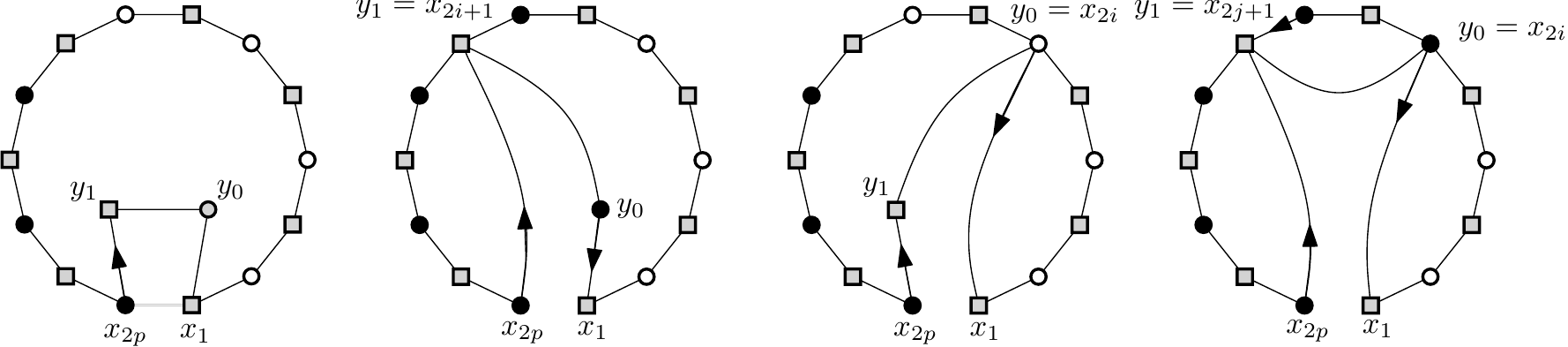}
\caption{\label{siteperc} This figure shows how percolation is sampled during the peeling process. The arrows on each
figure indicate the possible rerootings for the next peeling step. Note
that in the middle right case, if $y_0 = x_{2i}$ is white and if the quadrangulation on the right is infinite,
then a circuit of white vertices separates the root vertex from infinity, so
percolation does not occur.}
\end{center}
\end{figure}

\bigskip

Let us denote by $B_n$ the number of black vertices on $\partial \qua_n$, $W_n$
the number of white vertices, and by $\mathcal{F}_n$ the filtration generated by
$\qua_0, \qua_1 , \cdots,
\qua_n$ and their coloring. Recall that $X_n$ denotes the increment size of the
boundary length conditionally on $\mathcal{F}_n$, and that its distribution is
given by \eqref{stepsize1}, \eqref{stepsize2}. We now give the
explicit transition probabilities of $B_n$ conditionally on $\mathcal{F}_n$. In order to
simplify notation, we write $| \partial \qua_n |= 2 p$.

\bigskip

\begin{itemize}
 
\item[(1)] When $X_n = 1$, the face discovered has two new vertices, among them one
belonging to the UIPM, that gets color black or white (see Figure \ref{siteperc}, left for an illustration). Therefore,
\begin{equation*}
B_{n+1} = \begin{cases}
           B_n + 1 & \text{ with probability $q \frac{C_{p+1}}{12 C_p}$},\\
           B_n & \text{ with probability $(1 - q) \frac{C_{p+1}}{12 C_p}$}.
          \end{cases}
\end{equation*}

\end{itemize}

\bigskip

We now consider the event $X_n = -k \leqslant 0$ for some $k \in \{0,\ldots,p-1\}$, that is, some vertices are
removed from $\partial \qua_n$. Let us discuss the different cases that may occur, according to Section \ref{peeling}.

\begin{itemize}

\item[(2)] $y_0 \notin \partial \qua_n$ and $y_1 \in \partial \qua_n$ (Figure \ref{siteperc}, middle left). The
vertex $y_0$ belongs to the unexplored part of the UIPM, and it is colored black
or white (with the corresponding probabilities), independently of previously
chosen colors.

On the one hand, if the quadrangulation $\qua_n^l$ is infinite, then black vertices are removed
if and only if $p - k < B_n$, and in this case $\partial \qua_{n+1}$ has no white
vertices. If $p - k \geqslant B_n$, then no black vertex is removed and $B_{n+1} = B_n$. Hence, $B_{n+1} = \min \left( B_n, p-k \right)$ in this case.

On the other hand, if $\qua_n^r$ is infinite, then $|\partial \qua_n^l | = 2(k+1)$ and
the number of black vertices removed is $\min \left( B_n , k + 1
\right)$. In addition, one black vertex is added with probability $q$. This
gives
$B_{n+1} = \max \left( B_n - k - 1 , 0 \right) + 1$ with probability $q$, and
$B_{n+1} = \max \left( B_n - k - 1 , 0 \right)$ with probability $(1-q)$.

\item[(3)] $y_0 \in \partial \qua_n$ and $y_1 \notin \partial \qua_n$ (Figure \ref{siteperc}, middle right).
The situation is very similar to (2), except that no new colored vertex is
added. If the quadrangulation $\qua_n^l$ is infinite, then one has $B_{n+1} = \min
\left( B_n, p-k \right)$, and if $\qua_n^r$ is infinite, then $B_{n+1} = \max
\left( B_n - k , 0 \right)$.

\item[(4)] $y_0, y_1 \in \partial \qua_n$ (Figure \ref{siteperc}, right).
If $\qua_n^r$ is infinite, the
situation is identical to the corresponding case in (3) and $B_{n+1} = \max
\left( B_n - k , 0 \right)$, while if $\qua_n^l$ is infinite, the
situation is identical to the corresponding case in (2) and $B_{n+1} = \min
\left( B_n, p-k \right)$.

Finally, if $\qua_n^m$ is infinite, then there is $1 \leqslant i \leqslant k$
such that $|\partial \qua_n^l | = 2i $, and $B_{n+1} = \max \left( B_n - i , 0 \right)$.
\end{itemize}

For each of the previous cases, the corresponding probabilities have been
determined in Section \ref{peeling}. We deduce that conditionally on
$|\partial \qua_n| = 2p$, and when $X_n = -k$:
\begin{equation*}
B_{n+1} =
\begin{cases}
\min \left( B_n, p-k \right) & \text{w. p.
$2\frac{C_{p-k}Z_{k+1}}{12 C_p} + \frac{C_{p-k}}{12C_p} \sum_{i = 1}^k Z_i
Z_{k+1-i}$}, \\
\max \left( B_n - k , 0 \right) & \text{w. p.
$\frac{C_{p-k}Z_{k+1}}{12 C_p} + \frac{C_{p-k}}{12C_p} \sum_{i = 1}^k Z_i
Z_{k+1-i}$},\\
\max \left( B_n - k - 1 , 0 \right) + 1 & \text{w. p.
$q\frac{C_{p-k}Z_{k+1}}{12 C_p}$},\\
\max \left( B_n - k - 1 , 0 \right) & \text{w. p.
$(1 -q)\frac{C_{p-k}Z_{k+1}}{12 C_p}$},\\
\max \left( B_n - i , 0 \right) & \text{w. p.
$\frac{C_{p-k}}{12C_p} Z_i Z_{k+1-i}$ for $1 \leqslant i \leqslant k$}.
\end{cases}
\end{equation*}


\bigskip

\subsection{Derivation of \texorpdfstring{$p_c^{\textrm{site}}$}{pc}}
\label{sec:siteperc}

We now show that $p_c^{\textrm{site}} = 2/3$ a.s. We first prove that black vertices do not percolate when $q<2/3$, and then that they percolate when $q>2/3$. We denote by $C_{\infty}$ the event that the root vertex is in an infinite black cluster.

\bigskip

Let us first consider $q < 2/3$. We start by noting that
\begin{equation} \label{first_lemma}
P(C_{\infty} \cap \{B_n = 1 \: \text{infinitely often}\}) = 0,
\end{equation}
which follows from the observation that
$$P(C_{\infty}|B_n =1) \leq 1-c$$
for some universal constant $c > 0$. Indeed, if $B_n = 1$ and $X_n \leq -1$, then black vertices disappear on the next step with probability at least $1/2$. Hence, $B_{n+1}=0$ with probability at least
$$\frac{1}{2} P\big(X_n \leq -1\big) = \frac{1}{2} \Big(1 - \frac{4}{9} - \frac{3}{8} + o(1)\Big)$$
(using the distribution of $X$). This implies that
$$P(C_{\infty} \cap \{B_n = 1 \: \text{at least $k$ times}\}) \leq (1 - c)^k,$$
by conditioning on the first $k$ such times, and \eqref{first_lemma} follows readily.

\bigskip

We will now assume that $P(C_{\infty})>0$. As we have just observed, we can suppose that a.s., $B_n \geq 2$ for $n$ large enough. We introduce a modified Markov chain $(B'_n)$ obtained by ``simplifying'' $(B_n)$, in particular by allowing it to take negative values (and coupled in a natural way). More precisely, we consider the chain with the following transition probabilities, conditionally on $|\partial \qua_n| = 2p$:
\begin{equation*}
B'_{n+1} = \begin{cases}
           B'_n + 1 & \text{ with probability $q \frac{C_{p+1}}{12 C_p}$},\\
           B'_n & \text{ with probability $(1 - q) \frac{C_{p+1}}{12 C_p}$}
          \end{cases}
\end{equation*}
(corresponding to $X_n = 1$), and
\begin{equation*}
B'_{n+1} =
\begin{cases}
B'_n & \text{w. p.
$2\frac{C_{p-k}Z_{k+1}}{12 C_p} + \frac{C_{p-k}}{12C_p} \sum_{i = 1}^k Z_i
Z_{k+1-i}$}, \\
B'_n - k & \text{w. p.
$(1+q)\frac{C_{p-k}Z_{k+1}}{12 C_p} + \frac{C_{p-k}}{12C_p} \sum_{i = 1}^k Z_i
Z_{k+1-i}$},\\
B'_n - k - 1 & \text{w. p.
$(1 -q)\frac{C_{p-k}Z_{k+1}}{12 C_p}$},\\
B'_n - i & \text{w. p.
$\frac{C_{p-k}}{12C_p} Z_i Z_{k+1-i}$ for $1 \leqslant i \leqslant k$}
\end{cases}
\end{equation*}
for every $k=0,\ldots,p-1$ (corresponding to $X_n = -k$).

Now, let us note that the increment $(B'_{n+1} - B'_n)$ is equal to the increment $(B_{n+1} - B_n)$
except in the following three cases.
\begin{itemize}
\item $B_{n+1} = \min \left( B_n, p-k \right)$ and $B_n > p-k$: in this case,
 $$B_{n+1} - B_n = \min(B_n,p-k) - B_n = (p-k) - B_n < 0 = B'_{n+1} - B'_n.$$
\item $B_{n+1} = \max \left( B_n - k - 1 , 0 \right) + 1$ and $B_n -k-1 < 0$: in this case, $B_{n+1}=1$, which is ruled out by \eqref{first_lemma} (for $n$ large enough).
\item In each of the remaining three sub-cases, when $B_n -k < 0$, $B_n -k-1 < 0$, or $B_n - i < 0$ (resp.): this means that the number of black vertices gets negative, so that percolation does not occur.
\end{itemize}
Therefore, conditionally on $C_{\infty}$, one has $B_n \leqslant B_n' + O(1)$. We
will see that almost surely, $B'_n \to -\infty$,
and therefore there exists $n$ such that $B_n=0$. This will imply that the probability that percolation occurs is $0$. One has:

\begin{align*}
E & \left[B'_{n+1} - B'_n \middle| |\partial \qua_n| = 2p \right] \\
& = q P \left( X_n = 1 \middle| |\partial \qua_n| = 2p \right) 
 - \sum_{k = 0}^{p-1} k \left( 2\frac{C_{p-k}Z_{k+1}}{12 C_p} +
\frac{C_{p-k}}{12C_p} \sum_{i = 1}^k Z_i Z_{k+1-i} \right)\\
& \quad \quad - (1- q) \sum_{k=0}^{p-1} \frac{C_{p-k}Z_{k+1}}{12 C_p} -
\sum_{k=1}^{p-1}
\sum_{i = 1}^ k i \frac{C_{p-k}}{12C_p} Z_i Z_{k+1-i}\\
& = \left(q - \frac{1}{2}  \right) P \left( X_n = 1 \middle| |\partial \qua_n| =
2p
\right)  + \frac{1}{2} E \left[X_n \middle| |\partial \qua_n| = 2p \right] \\
& \quad \quad - (1- q) \sum_{k=0}^{p-1} \frac{C_{p-k}Z_{k+1}}{12 C_p} +
\sum_{k=1}^{p-1}
\sum_{i = 1}^ k \left(\frac{k}{2} - i \right) \frac{C_{p-k}}{12C_p} Z_i
Z_{k+1-i}\\
& = \left(q - \frac{1}{2}  \right) \frac{C_{p+1}}{12 C_p}  + \frac{1}{2} E
\left[X_n \middle| |\partial \qua_n| = 2p \right] \\
& \quad \quad - (1- q) \sum_{k=0}^{p-1} \frac{C_{p-k}Z_{k+1}}{12 C_p} -
\frac{1}{2}
\sum_{k=1}^{p-1} \sum_{i = 1}^ k  \frac{C_{p-k}}{12C_p} Z_i
Z_{k+1-i}.
\end{align*}
Corollary \ref{expectX}, and the computations performed in its proof, ensure that a.s.
\[
E \left[B'_{n+1} - B'_n \middle| \mathcal{F}_n \right]
\longrightarrow
\left(q - \frac{1}{2}  \right)\frac{3}{8} - (1-q) \frac{1}{8} -
\frac{1}{2}\frac{1}{24} = \frac{q}{2} - \frac{1}{3}
\]
as $n \to \infty$. Therefore, $E \left[B'_{n+1} - B'_n \middle|
\mathcal{F}_n \right]$ is negative and bounded away from $0$ for
$n$ large enough. This suffices to prove that $B'_n \to - \infty$ almost
surely, and percolation does not occur.

\bigskip

Now, let us take a value $q > 2/3$. As mentioned earlier, one cannot simply exchange the roles
of black and white sites to prove that $W_n$ stays small, and that consequently black vertices percolate. However, using $X_n = (W_{n+1} - W_n) + (B_{n+1} - B_n)$, we can obtain: conditionally on $|\partial \qua_n| = 2p$,
\begin{equation*}
W_{n+1} =
\begin{cases}
W_n +1 & \text{w. p. $(1-q) \frac{C_{p+1}}{12 C_p}$}\\
W_n & \text{w. p. $q \frac{C_{p+1}}{12 C_p}$}\\
\max \left( W_n - k, 0 \right) & \text{w. p.
$2\frac{C_{p-k}Z_{k+1}}{12 C_p} + \frac{C_{p-k}}{12C_p} \sum_{i = 1}^k Z_i
Z_{k+1-i}$}, \\
\min \left( W_n, p - k \right) & \text{w. p.
$\frac{C_{p-k}Z_{k+1}}{12 C_p} + \frac{C_{p-k}}{12C_p} \sum_{i = 1}^k Z_i
Z_{k+1-i}$},\\
\min \left( W_n, p - k - 1 \right) & \text{w. p.
$q\frac{C_{p-k}Z_{k+1}}{12 C_p}$},\\
\min \left( W_n , p - k - 1 \right) +1 & \text{w. p.
$(1 -q)\frac{C_{p-k}Z_{k+1}}{12 C_p}$},\\
\max \left( W_n - (i -1) , 0 \right) & \text{w. p.
$\frac{C_{p-k}}{12C_p} Z_i Z_{k+1-i}$ for $1 \leqslant i \leqslant k$}.
\end{cases}
\end{equation*}

In a similar way as for $B'_n$, we consider the process $W'_n$, coupled with $W_n$
and with increments given conditionally on $|\partial \qua_n| = 2p$ by :
\begin{equation*}
W'_{n+1} =
\begin{cases}
W'_n +1 & \text{w. p. $(1-q) \frac{C_{p+1}}{12 C_p} + (1 - q)
\frac{C_{p-k}Z_{k+1}}{12 C_p}$}\\
W'_n & \text{w. p. $q \frac{C_{p+1}}{12 C_p} + (1+ q) \frac{C_{p-k}Z_{k+1}}{12
C_p} + \frac{C_{p-k}}{12C_p} \sum_{i = 1}^k Z_i Z_{k+1-i}$}\\
W'_n - k & \text{w. p.
$2\frac{C_{p-k}Z_{k+1}}{12 C_p} + \frac{C_{p-k}}{12C_p} \sum_{i = 1}^k Z_i
Z_{k+1-i}$}, \\
W'_n - (i -1) & \text{w. p.
$\frac{C_{p-k}}{12C_p} Z_i Z_{k+1-i}$ for $1 \leqslant i \leqslant k$}.
\end{cases}
\end{equation*}
Then, conditionally on the event $\{W_n > 0, \forall n \geqslant 0 \}$, the
increments of $W'_n$ are bigger than the increments of $W_n$. As $n
\to \infty$, one has
\[
 E \left[W'_{n+1} - W'_n \middle| \mathcal{F}_n \right]
\longrightarrow
\frac{1}{3} - \frac{q}{2} := - \alpha < 0,
\]
from which one can easily deduce that a.s. $W_n = O(\ln n)$: we now provide an explicit proof for the sake of completeness. If we write $\Delta_n = W'_{n+1} - W'_n$, we obtain, for $n \geq N$,
$$E \big[\Delta_n | \mathcal{F}_n \big] \leq - \frac{\alpha}{2}.$$
For any fixed $n \geq m \geq N$, one has
$$P \bigg( \sum_{l=m}^n \Delta_l > C \ln n | \mathcal{F}_m \bigg) \leq e^{-\lambda C \ln n} E \bigg[ \exp \Big( \lambda \sum_{l=m}^n \Delta_l \Big) \big| \mathcal{F}_m \bigg].$$
We can write
$$E \bigg[ \exp \Big( \lambda \sum_{l=m}^n \Delta_l \Big) \big| \mathcal{F}_m \bigg] = E \bigg[ \exp \Big( \lambda \sum_{l=m}^{n-1} \Delta_l \Big) E[e^{\lambda \Delta_n} | \mathcal{F}_{n-1} ] \big| \mathcal{F}_m \bigg],$$
and use
\begin{equation*}
E[e^{\lambda \Delta_n} | \mathcal{F}_{n-1} ]
\leq 1 + C_1 \lambda E[\Delta_n | \mathcal{F}_{n-1}] + C_2 |\lambda|^{5/4} E[|\Delta_n|^{5/4} | \mathcal{F}_{n-1}]
\leq 1
\end{equation*}
if we choose a $\lambda >0$ small enough (we used $E[|\Delta_n|^{5/4} | \mathcal{F}_{n-1}] \leq M$, for some universal constant $M$: this follows from the fact that the probabilities $P \left( X_n = -k \big| |\partial \qua_n | = 2p
\right)$ are increasing in $p$ and converge to $q_k$, which is of order $q_k \sim c k^{-5/2}$ -- using \eqref{Z_p} and \eqref{q_k}). By iterating this reasoning, we find
$$P \bigg( \sum_{l=m}^n \Delta_l > C \ln n | \mathcal{F}_m \bigg) \leq e^{-\lambda C \ln n},$$
which allows one to conclude, by using a Borel Cantelli argument (choosing a large enough $C$).
Since $B_n + W_n = |\partial \qua_n | \approx n^{2/3}$, we deduce that $B_n \approx n^{2/3}$: in particular, black vertices percolate.

\section{Bond percolation on the UIPM} \label{bond_percolation}

In this section, we study bond percolation, instead of site
percolation, on the UIPM: each edge is open with probability $q$, and closed with
probability $(1-q)$, independently of other edges. We prove the second part of Theorem \ref{main_theorem}: the corresponding percolation threshold is almost surely
$$p_c^{\textrm{bond}}=1/2.$$

\subsection{Exploration process}

In this section, we describe how to sample bond percolation on the UIPM
simultaneously with a peeling process of the UIPQ. This is similar to the
exploration process for site percolation described in Section \ref{expsite},
but small adaptations are needed for the process to actually follow the boundary of the
percolation cluster of the root vertex. We will assume for simplicity that the root edge of
$\map_{\infty}$ is open.

Let us consider the UIPM $\map_{\infty}$, and $\qua_{\infty}$ the associated
UIPQ.
Let us denote by $\qua_0, \qua_1, \ldots, \qua_n$ the peeling process for
$\qua_{\infty}$, and $\map_0, \map_1, \ldots, \map_n$ the associated submaps of
$\map_{\infty}$. Each time a new face of $\qua_{\infty}$ is discovered, the
corresponding edge of $\map_{\infty}$ is opened with probability $q$, and closed
with probability $(1-q)$ independently of all previous steps. The percolation
interfaces between open and closed edges can be viewed as a random tiling of
$\qua_{\infty}$, as illustrated in Figure \ref{tiling}.

\begin{figure}[ht!]
\begin{center}
\includegraphics[width=7cm]{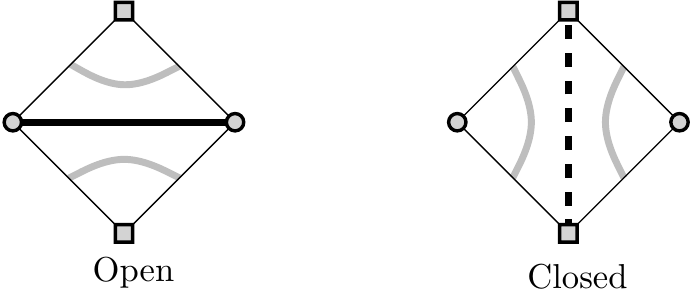}
\caption{\label{tiling}The exploration process can be seen as a random tiling of the quadrangles that are successively discovered.}
\end{center}
\end{figure}

It is possible to adapt the peeling process in order to follow
percolation interfaces. Let $\map_n^0$ denote the set of vertices connected to the
root vertex of $\map_n$ by open paths lying in $\map_n$: this is the part of the cluster of the root $\map_{\infty}^0$ discovered before time $n$ with
the peeling procedure. The choice of the next quadrangle
to reveal is
very similar to what we did for site percolation. Recall that on the quadrangulation, circle-vertices belong to the associated map, while square-vertices lie on the dual of this map. On the one hand, if all circle-vertices of
$\partial \qua_n$ belong to $\map_n^0$, or if, on the contrary, no circle-vertex of
$\partial \qua_n$ belongs to $\map_n^0$, then we can make an arbitrary choice for the next step. On the other hand,
if some, but not all, circle-vertices of $\partial \qua_n$ belong to $\map_n^0$, then we can find three vertices $x_{2p}$, $x_1$, $x_2$ (in this order) such that $x_{2p}$ belongs to $\map_n^0$, but not $x_2$ (see Figure \ref{bond}): we reveal the quadrangle on the left side of the edge $(x_{2p} x_1)$. Provided that this
procedure is followed during the peeling process, then the vertices of
$\partial \qua_n \cap \map_n^0$ form an arc of $\partial \qua_n$.

\begin{figure}[h!]
\begin{center}
\includegraphics[width=13cm]{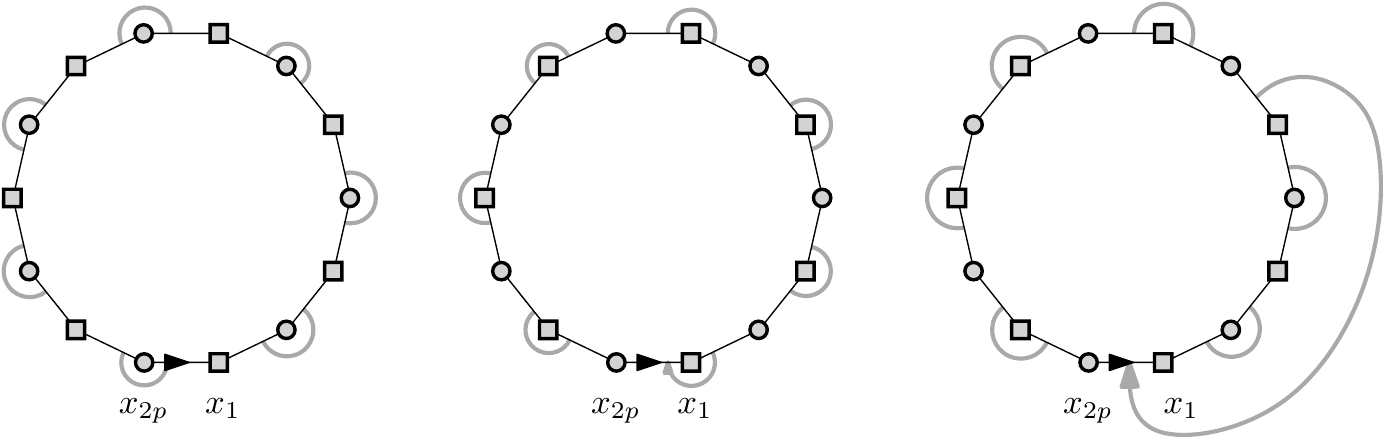}
\caption{\label{bond} Configurations obtained by following the exploration
process during the peeling procedure. \emph{Left:} No vertex of the boundary belongs to the
explored part of the cluster of the root, so that percolation does not occur. \emph{Middle:} All
vertices of the boundary belong to the explored part of the cluster of the root. \emph{Right:}
The vertices on the left belong to the explored part of the cluster of the root, whereas
the vertices on the right do not belong to the discovered part of the cluster of the root.}
\end{center}
\end{figure}
Now, let $A_n$ denote the number of vertices of $\partial \qua_n$
that belong to $\map_n^0$. If there exists $n$ such that $A_n = 0$, then
the root vertex does not percolate, and $A_k = 0$ for all $k \geqslant n$. On the other hand,
if $(A_n)$ is unbounded, then percolation does occur. Let $\mathcal{F}_n$ denote the
filtration generated by $\qua_0, \ldots, \qua_n$ and bond percolation on them.
Let $n > 0$, and suppose that $A_n > 0$. Following a similar strategy as for site percolation, we give
explicit transition probabilities for $A_n$ conditionally on $\mathcal{F}_n$. Recall that $X_n$ denotes
the increment size for the boundary length conditionally on $\mathcal{F}_n$, and that
its distribution is given by \eqref{stepsize1}, \eqref{stepsize2}. Let us also set $| \partial \qua_n |= 2 p$, as before. 

\bigskip

\begin{itemize}

\item[(1)] When $X_n = 1$, let us denote by $(x_{2p},x_1,y_0,y_1)$ the face discovered: it has two new vertices, and a new edge $(x_{2p},y_0)$ of the UIPM. With
probability $q$, the new edge is open and there is an open path joining
$y_0$ to
the root vertex in $\map_{n+1}$. With probability $(1-q)$, this edge is closed
and $y_0$
does not belong to the part discovered of the cluster of the root (note that $y_0$ may still
belong to the cluster of the root, if some of the edges that connect it to the root
have not yet been discovered). This yields
\begin{equation*}
A_{n+1} = \begin{cases}
           A_n + 1 & \text{ with probability $q \frac{C_{p+1}}{12 C_p}$},\\
           A_n & \text{ with probability $(1 - q) \frac{C_{p+1}}{12 C_p}$}
          \end{cases}
\end{equation*}
(see Figure \ref{bond1} for an illustration).

\end{itemize}

\begin{figure}[ht!]
\begin{center}
\includegraphics[width=13cm]{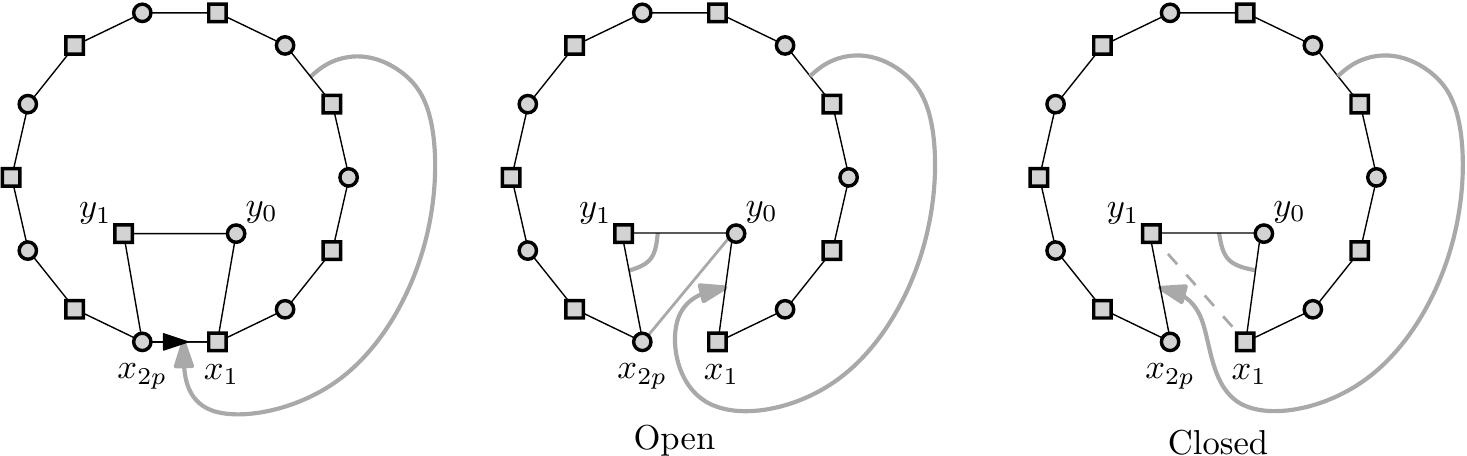}
\caption{\label{bond1} Evolution of the exploration process in case (1), when two new
vertices are discovered.}
\end{center}
\end{figure}

\bigskip

Suppose now that $X_n = -k$. We discuss the different cases that appeared in Section \ref{peeling}.

\begin{itemize}

\item[(2)] $y_0 \notin \partial \qua_n$ and $y_1 \in \partial \qua_n$ (see Figure \ref{bond2}). The situation is somewhat
similar to site percolation, except that the vertices of $\partial \qua_n
\cap \map_n$ that do not belong to $\map_n^0$ may still be connected to it by
 not-yet-discovered open edges. We claim that, except for $y_0$, a vertex of
$\partial \qua_{n+1} \cap \map_{\infty}$ belongs to $\map_{n+1}^0$ if and only if
it belongs to $\map_{n}^0$. Indeed, the two parts $\qua_n^l$ and $\qua_n^r$
can only be connected by the new edge or by vertices of $\map_n$, therefore,
filling the finite part with a mix of open and closed edges will not change whether vertices on the boundary of the infinite one belong or not to
$\map_{n}^0$. This gives the same transitions as for site percolation:
\begin{equation*}
A_{n+1} =
\begin{cases}
\min \left( A_n, p-k \right) & \text{ (if $\qua_n^l$ is infinite),} \\
\max \left( A_n - k - 1 , 0 \right) + 1 & \text{ with probability
$q$ (if $\qua_n^r$ is infinite),}\\
\max \left( A_n - k - 1 , 0 \right) & \text{ with probability
$(1 -q)$ (if $\qua_n^r$ is infinite).}
\end{cases}
\end{equation*}
\begin{figure}[ht!]
\begin{center}
\includegraphics[width=13cm]{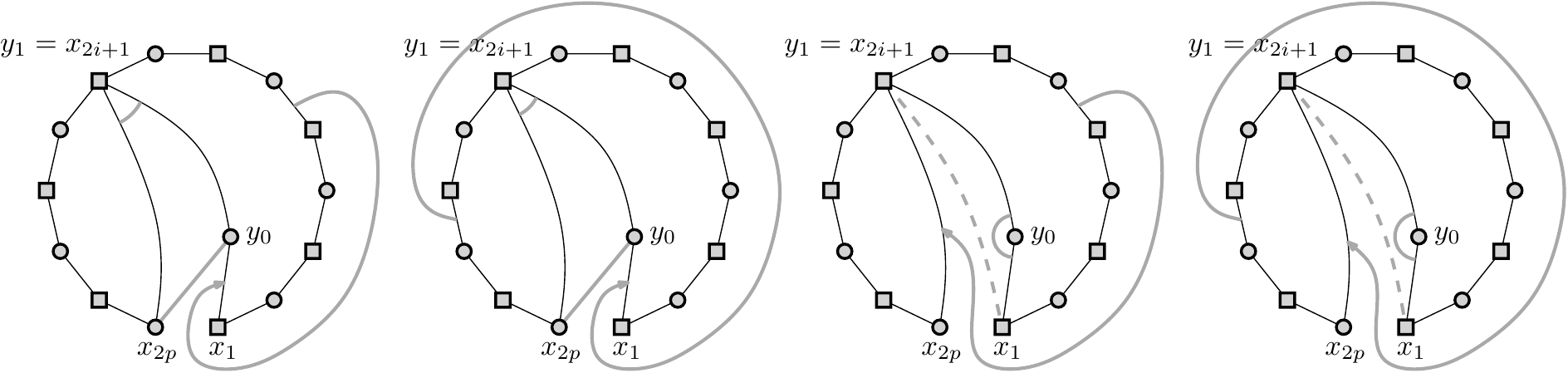}
\caption{\label{bond2} Evolution of the exploration process in case (2).}
\end{center}
\end{figure}

\bigskip

\item[(3)] $y_0 \in \partial \qua_n$ and $y_1 \notin \partial \qua_n$. Here the situation is similar, except for a notable difference when $y_0 \notin \map_n^0$. Indeed, in this case one can have $y_0 \in \map_{n+1}^0$ even if the new bond is closed. This happens when $\qua_n^r$ is infinite:
filling e.g. $\qua_n^l$ with open edges connects $y_0$ to the root
vertex by a path of open edges belonging to $\qua_{n+1}$ as long as there is at
least one vertex of $\partial \qua_n$ that belongs to $\map_n^0$. On the other
hand, filling $\qua_n^l$ with closed edges leaves $y_0$ disconnected from the
root vertex in $\qua_{n+1}$, and percolation does not occur (see Figure
\ref{bond3} for an illustration). The corresponding probabilities depend on $\qua_n^l$, but their exact values will not be needed. Note that if
$\qua_n^l$ is infinite, then $y_0$ stays disconnected from the root vertex in
$\qua_{n+1}$ if $y_0 \notin \map_n^0$.
\begin{figure}[ht!]
\begin{center}
\includegraphics[width=13cm]{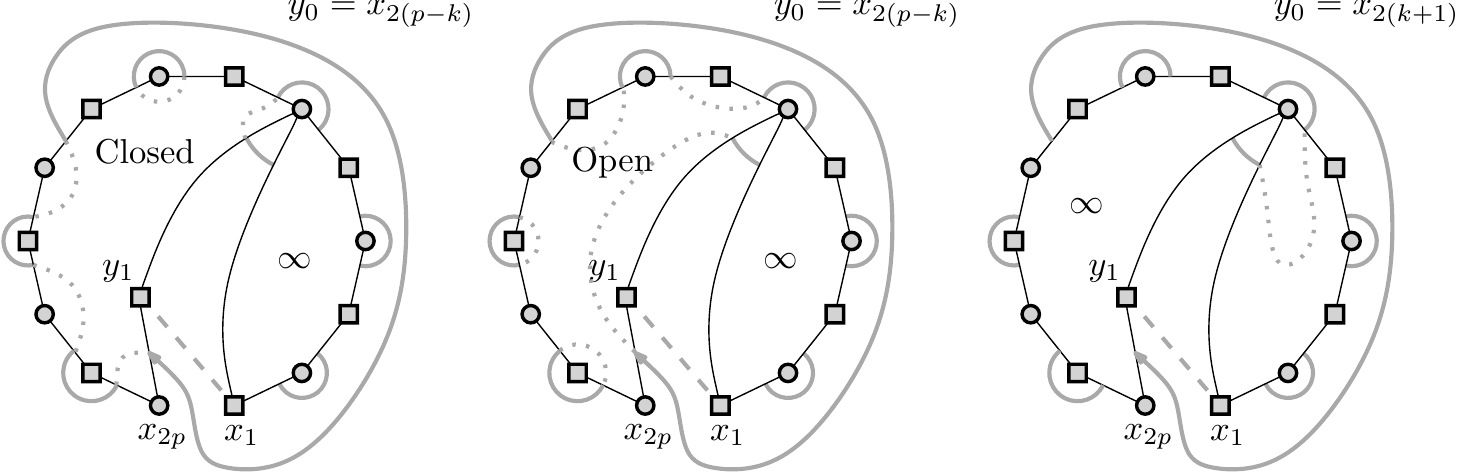}
\caption{\label{bond3} Evolution of the exploration process in case (3),
when $y_0 \notin \map_n^0$. \emph{Left:} $y_0 \notin \map_{n+1}^0$ and percolation
does not occur. \emph{Middle:} $y_0 \in \map_{n+1}^0$. \emph{Right:} $y_0 \notin
\map_{n+1}^0$.}
\end{center}
\end{figure}

The transitions are thus given by:
\begin{equation*}
A_{n+1} =
\begin{cases}
\min \left( A_n, p-k \right) & \text{ with
probability $(1-q)$ (if $\qua_n^l$ is infinite),} \\
\min \left( A_n + 1, p-k \right) & \text{ with
probability $q$ (if $\qua_n^l$ is infinite),} \\
A_n - k & \text{ if $A_n - k > 0$ (if $\qua_n^r$ is infinite),}\\
 0 & \text{ with probability
$> 0$ if $A_n - k \leqslant 0$ (if $\qua_n^r$ is infinite),}\\
 1 & \text{ with probability
$> 0$ if $A_n - k \leqslant 0$ (if $\qua_n^r$ is infinite).}
\end{cases}
\end{equation*}

\bigskip

\item[(4)] $y_0, y_1 \in \partial \qua_n$. If $\qua_n^l$ is infinite, then the situation is simple and a vertex of $\partial
\qua_{n+1}$ belongs to $\map_{n+1}^0$ \emph{iff} it belongs to $\map_n^0$. We thus obtain $A_{n+1} = \min \left( A_n, p-k \right)$ in this case.

If $\qua_n^r$ is infinite, then the situation is identical to case (3) (when $\qua_n^r$ is infinite), which gives
\begin{equation*}
A_{n+1} =
\begin{cases}
A_n - k & \text{ if $A_n - k > 0$,}\\
 0 & \text{ with probability
$> 0$ if $A_n - k \leqslant 0$,}\\
 1 & \text{ with probability
$> 0$ if $A_n - k \leqslant 0$.}
\end{cases}
\end{equation*}

Finally, when $\qua_n^m$ is infinite, let us write $y_0 = x_{2i}$ ($1 \leqslant i \leqslant k$). If $p-i \leqslant A_n - 1$, then every vertex of
$\partial \qua_n^m \cap \map_{\infty}$ is in $\map_n^0$. In this case we have
$A_{n+1} = p-k$.
Suppose now that $p-i \geqslant A_n$. The $(k-i+1)$ circle-vertices of $\partial
\qua_n^l$ are not in $\partial \qua_{n+1}$, which means that $\max \left( A_n -k
+ i - 1,0 \right)$ vertices in $\partial \qua_{n+1}$ belong to $\map_n^0$. These vertices also belong to
$\map_{n+1}^0$, and in addition, the vertex $y_0$ belongs to $\map_{n+1}^0$ \emph{iff}
the new edge is open. To sum up, the transitions in this final situation are:
\begin{equation*}
A_{n+1} =
\begin{cases}
p - k & \text{ if $A_n > p-i$,}\\
\max \left( A_n - k + i - 1 , 0 \right) + 1 & \text{ with probability
$q$ if $A_n \leqslant p-i$,}\\
\max \left( A_n - k + i - 1 , 0 \right) & \text{ with probability
$(1-q)$ if $A_n \leqslant p-i$.}\\
\end{cases}
\end{equation*}

\end{itemize}

\bigskip

\subsection{Derivation of \texorpdfstring{$p_c^{\textrm{bond}}$}{pc}}

Suppose now $q < 1/2$, and consider the modified Markov chain $(A'_n)$ with conditional transition probabilities  given $\mathcal{F}_n$: if $X_n = 1$,
\begin{equation}
A'_{n+1} = \begin{cases}
           A'_n + 1 & \text{ with probability $q \frac{C_{p+1}}{12 C_p}$,}\\
           A'_n & \text{ with probability $(1 - q) \frac{C_{p+1}}{12 C_p}$.}
          \end{cases}
\end{equation}
If $X_n = - k$, we set
\begin{equation}
A'_{n+1} =
\begin{cases}
A'_n & \text{ with probability
$(2-q) \frac{C_{p-k}Z_{k+1}}{12 C_p} +
\frac{C_{p-k}}{12C_p} \sum_{i = 1}^k Z_i
Z_{k+1-i}$}, \\
A'_n + 1 & \text{ with probability
$q\frac{C_{p-k}Z_{k+1}}{12 C_p}$} \\
\end{cases}
\end{equation}
(corresponding to $\qua_n^l$ infinite),
\begin{equation}
A'_{n+1} =
\begin{cases}
A'_n - k & \text{ with probability
$(1+q)\frac{C_{p-k}Z_{k+1}}{12 C_p} + \frac{C_{p-k}}{12C_p} \sum_{i = 1}^k Z_i
Z_{k+1-i}$},\\
A'_n - k - 1 & \text{ with probability
$(1 -q)\frac{C_{p-k}Z_{k+1}}{12 C_p}$}\\
\end{cases}
\end{equation}
(corresponding to $\qua_n^r$ infinite), and
\begin{equation}
A'_{n+1} =
\begin{cases}
A'_n - k + i & \text{ with probability
$q \frac{C_{p-k}}{12C_p} Z_i Z_{k+1-i}$ for $1 \leqslant i \leqslant k$},\\
A'_n - k + i - 1 & \text{ with probability
$(1-q) \frac{C_{p-k}}{12C_p} Z_i Z_{k+1-i}$ for $1 \leqslant i \leqslant k$}
\end{cases}
\end{equation}
(corresponding to $\qua_n^m$ infinite).

\bigskip

In a similar way as for site percolation, we can write
\begin{align*}
E & \left[A'_{n+1} - A'_n \middle| |\partial \qua_n| = 2p \right]\\
& = q P \left( X_n = 1 \middle| |\partial \qua_n| = 2p \right) 
 - \sum_{k = 0}^{p-1} k \left( \frac{1}{2} P \left( X_n = -k \middle|
|\partial \qua_n| = 2p \right) - \frac{1}{2}
\sum_{i = 1}^ k \frac{C_{p-k}}{12C_p} Z_i Z_{k+1-i}
\right)\\
& \quad \quad - (1- q) \sum_{k=0}^{p-1} \frac{C_{p-k}Z_{k+1}}{12 C_p} 
+ q \sum_{k=0}^{p-1} \frac{C_{p-k}Z_{k+1}}{12 C_p}\\
& \quad \quad + \sum_{k=1}^{p-1}
\sum_{i = 1}^ k (-k + i) \frac{C_{p-k}}{12C_p} Z_i Z_{k+1-i}
- (1-q) \sum_{k=1}^{p-1}
\sum_{i = 1}^ k \frac{C_{p-k}}{12C_p} Z_i Z_{k+1-i}
\\
& = \left(q - \frac{1}{2}  \right) P \left( X_n = 1 \middle| |\partial \qua_n| = 2p
\right)  + \frac{1}{2} E \left[X_n \middle| |\partial \qua_n| = 2p \right] \\
& \quad \quad + (2q - 1) \sum_{k=0}^{p-1} \frac{C_{p-k}Z_{k+1}}{12 C_p} +
\sum_{k=1}^{p-1}
\sum_{i = 1}^ k \left(- \frac{k}{2} + i \right) \frac{C_{p-k}}{12C_p} Z_i
Z_{k+1-i}\\
& \quad \quad + (q - 1) \sum_{k=0}^{p-1} \sum_{i = 1}^k \frac{C_{p-k}}{12C_p} Z_i
Z_{k+1-i}
\\
& = \left(q - \frac{1}{2}  \right) P \left( X_n = 1 \middle| |\partial \qua_n| = 2p
\right)  + \frac{1}{2} E \left[X_n \middle| |\partial \qua_n| = 2p \right] \\
& \quad \quad + (2q - 1) \sum_{k=0}^{p-1} \frac{C_{p-k}Z_{k+1}}{12 C_p} +  \left(q -
\frac{1}{2} \right) \sum_{k=0}^{p-1} \sum_{i = 1}^k \frac{C_{p-k}}{12C_p} Z_i
Z_{k+1-i},
\end{align*}
which is negative and stays bounded away from $0$ as $n \to \infty$. Using a domination of $A_n$ by $A'_n$ as we did for site percolation, we deduce that percolation does not occur a.s., and $A_n$ ``stays small''. Here we can then use directly a symmetry argument, and deduce that $p_c= 1/2$ a.s.

\subsection{Bond percolation on quadrangulations}

As a conclusion, we would like to mention that the previous reasoning can easily be adapted to study bond percolation on various classes of maps, in particular on $p$-angulations, as soon as one has counting formulas such as \eqref{counting_quad} at one's disposal. 

\begin{figure}[ht!]
\begin{center}
\includegraphics[width=13cm]{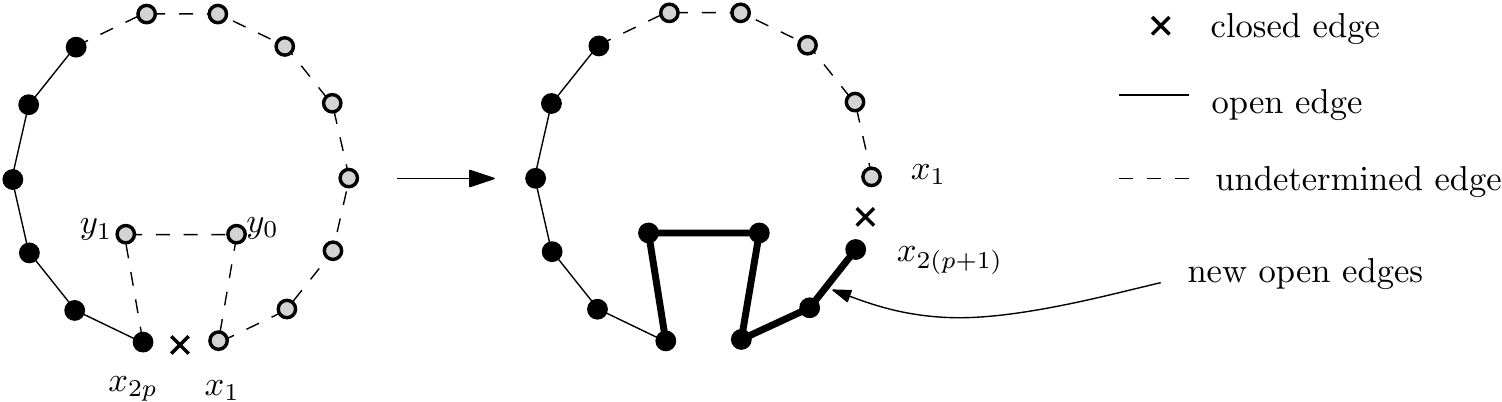}
\caption{\label{bondquad}Exploration process for bond percolation on the UIPQ. At each step, we explore iteratively the available ``undetermined'' edges until we find a closed one. The remaining boundary edges are then left undetermined.}
\end{center}
\end{figure}

For example, the previous peeling process can be used for bond percolation on the UIPQ: we now describe explicitly the exploration process in this case. We
consider percolation with parameter $q$, and will follow the boundary of
a cluster of open edges by exploring only the neighboring closed
edges, and leaving ``undetermined'' the remaining ones. More
precisely, conditionally on $|\partial \qua_n | = 2p$, the boundary $\partial \qua_n$ will consist in this case of
a certain number $A_n$ of vertices of the UIPQ, connected by $(A_n - 1)$
open edges, $1$ closed edge, and $U_n = 2p - A_n$ undetermined edges. Note
that $U_n$ also counts the number of ``free'' vertices that can get
connected in a later step to the open cluster that we are following.

At each step, we then reveal a quadrangle $(x_{2p}, x_1, y_0, y_1)$ as
before, lying on the left hand side of the unique closed edge $e=
(x_{2p},x_1)$, and we explore successively the undetermined edges
following $e$ on $\partial \qua_{n+1}$, until we find a closed one (or
no undetermined edge remains, in which case we can consider the edge explored last to be closed without any loss of generality). A certain number of new vertices get
connected in this way, which follows a geometric distribution with
parameter $q$ truncated by the number of undetermined edges $N$: let us
introduce the notation $G_q(N)$ for such a distribution (i.e. $P(G_q(N) = k) = q^k(1-q)$ for $0
\leq k < N$, and $= q^N$ for $k=N$).

\bigskip

\begin{itemize}

\item[(1)] When $y_0,y_1 \notin \partial \qua_n$, i.e. $X_n = 1$, we
simply have $U_n+2 = 2p - A_n+2$ free vertices at our disposal. This
yields
$$A_{n+1} = A_n \stackrel{(\ci)}{+} G_q(2p-A_n+2)$$
(see Figure \ref{bondquad} for an illustration).

\end{itemize}

\bigskip

Let us now assume that $X_n = -k$.

\begin{itemize}

\item[(2)] In the case when $y_0 \notin \partial \qua_n$ and $y_1 \in
\partial \qua_n$, we obtain:
\begin{itemize}
\item if $\qua_n^l$ is infinite,
\begin{equation*}
A_{n+1} =
\begin{cases}
A_n \stackrel{(\ci)}{+} G_q(2(p-k)-A_n) & \text{ (if $A_n < 2(p-k)$)},\\
2(p-k) & \text{ (if $A_n \geq 2(p-k)$)},
\end{cases}
\end{equation*}

\item if $\qua_n^r$ is infinite,
\begin{equation*}
A_{n+1} =
\begin{cases}
\big[A_n - (2k+1)\big] \stackrel{(\ci)}{+} G_q(2p-A_n+1) & \text{
(if $A_n \geq 2k+2$)},\\
\\
0 & \text{ w.p. $> 0$ (if $A_n < 2k+2$)},\\
1 \stackrel{(\ci)}{+} G_q(2(p-k)-1) & \text{ w.p. $> 0$ (if $A_n < 2k+2$).}
\end{cases}
\end{equation*}
\end{itemize}

\bigskip

\item[(3)] In the case when $y_0 \in \partial \qua_n$ and $y_1 \notin
\partial \qua_n$, we obtain:
\begin{itemize}
\item if $\qua_n^l$ is infinite,
\begin{equation*}
A_{n+1} =
\begin{cases}
A_n \stackrel{(\ci)}{+} G_q(2(p-k)-A_n) & \text{ (if $A_n < 2(p-k)-1$)},\\
2(p-k)-1 \stackrel{(\ci)}{+} G_q(1) & \text{ (if $A_n \geq 2(p-k)-1$)},
\end{cases}
\end{equation*}

\item if $\qua_n^r$ is infinite,
\begin{equation*}
A_{n+1} =
\begin{cases}
\big[A_n - 2k\big] \stackrel{(\ci)}{+} G_q(2p-A_n) & \text{ (if $A_n
\geq 2k+1$)},\\
\\
0 & \text{ w.p. $> 0$ (if $A_n < 2k+1$)},\\
1 \stackrel{(\ci)}{+} G_q(2(p-k)-1) & \text{ w.p. $> 0$ (if $A_n < 2k+1$).}
\end{cases}
\end{equation*}
\end{itemize}

\bigskip

\item[(4)] In the case when $y_0, y_1 \in \partial \qua_n$, we obtain
the same transitions as in case (2) when $\qua_n^l$ is infinite, and
as in case (3) when $\qua_n^r$ is infinite. Finally, if $\qua_n^m$ is
infinite, let us write $y_0 = x_{2i}$ ($1 \leq i \leq k$). Then
\begin{equation*}
A_{n+1} =
\begin{cases}
2(p-k) & \text{ (if $A_n > 2(p-i)$)},\\
\\
\big[A_n - 2(k-i)-1\big] \\
\stackrel{(\ci)}{+} G_q(2(p-i)+1-A_n) &
\text{ (if $2(k-i)+2 \leq A_n \leq 2(p-i)$)},\\
\\
0 & \text{ w.p. $> 0$ (if $A_n < 2(k-i)+2$)},\\
1 \stackrel{(\ci)}{+} G_q(2(p-k)-1) & \text{ w.p. $> 0$ (if $A_n <
2(k-i)+2$)}.
\end{cases}
\end{equation*}

\end{itemize}

We can prove, in the same way as for site and bond percolation on the
UIPM, that if percolation occurs, then we fall only finitely many times into one of the cases when one returns to
$0$ or $1$ before exploring undetermined edges (if we return at a certain
time $n$, then $A_{n+1} = 0$ with a probability at least $c$, for some
universal constant $c>0$).

We now prove that $q_c = 1/3$. Let us first assume $q<1/3$, and dominate $A_n$ by $A'_n$ obtained by replacing all truncated geometric distributions by non-truncated ones (denoted by $\mathcal{G}_q$ in the following), and allowing it to take negative values as before. We first have, when $X_n=1$,

\begin{equation}
A'_{n+1} = A'_n \stackrel{(\ci)}{+} \mathcal{G}_q \quad \text{with probability $\frac{C_{p+1}}{12 C_{p}}$}.
\end{equation}
If $X_n = - k$, we set
\begin{equation}
A'_{n+1} =
A'_n \stackrel{(\ci)}{+} \mathcal{G}_q \quad \text{ with probability
$2 \frac{C_{p-k}Z_{k+1}}{12 C_{p}} +
\frac{C_{p-k}}{12C_{p}} \sum_{i = 1}^k Z_i
Z_{k+1-i}$}
\end{equation}
(corresponding to $\qua_n^l$ infinite),
\begin{equation}
A'_{n+1} =
\begin{cases}
A'_n - 2 k  \stackrel{(\ci)}{+} \mathcal{G}_q & \text{ with probability
$\frac{C_{p-k}Z_{k+1}}{12 C_{p}} + \frac{C_{p-k}}{12 C_{p}} \sum_{i = 1}^k Z_i
Z_{k+1-i}$},\\
A'_n - (2 k + 1) \stackrel{(\ci)}{+} \mathcal{G}_q & \text{ with probability
$\frac{C_{p-k}Z_{k+1}}{12 C_{p}}$}\\
\end{cases}
\end{equation}
(corresponding to $\qua_n^r$ infinite), and
\begin{equation}
A'_{n+1} =
A'_n -  2(k - i) - 1 \stackrel{(\ci)}{+} \mathcal{G}_q  \text{ with probability
$\frac{C_{p-k}}{12C_{p}} Z_i Z_{k+1-i}$ for $1 \leqslant i \leqslant k$}
\end{equation}
(corresponding to $\qua_n^m$ infinite). We can then write
\begin{align*}
E & \left[A'_{n+1} - A'_n \middle| \left| \partial \qua_n \right| = 2p \right]\\
& = \frac{q}{1-q} \left( \frac{C_{p+1}}{12 C_p} + \sum_{k=0}^{p-1} \left( 2 \frac{C_{p-k}Z_{k+1}}{12 C_p} + \sum_{i = 1}^ k \frac{C_{p-k}}{12C_p} Z_i Z_{k+1-i} \right) \right) \\
& \quad \quad + \sum_{k=0}^{p-1} \left( -(2k+1) + \frac{q}{1-q} \right)  \frac{C_{p-k}Z_{k+1}}{12 C_p} \\
& \quad \quad + \sum_{k=0}^{p-1} \left( -2k + \frac{q}{1-q} \right) \left( \frac{C_{p-k}Z_{k+1}}{12 C_p} + \sum_{i = 1}^ k \frac{C_{p-k}}{12C_p} Z_i Z_{k+1-i}\right) \\
& \quad \quad + \sum_{k=1}^{p-1} \sum_{i = 1}^ k \left( - 2(k - i) - 1  + \frac{q}{1-q} \right)  \frac{C_{p-k}}{12C_p} Z_i Z_{k+1-i}\\
& = \frac{q}{1-q} + \sum_{k=0}^{p-1} -k \left(4  \frac{C_{p-k}Z_{k+1}}{12 C_p} + 3 \sum_{i = 1}^ k \frac{C_{p-k}}{12C_p} Z_i Z_{k+1-i}\right) - \sum_{k=0}^{p-1}  \frac{C_{p-k}Z_{k+1}}{12 C_p}\\
& \quad \quad + \sum_{k=1}^{p-1} \sum_{i = 1}^ k \left(2i - (k+1) \right)  \frac{C_{p-k}}{12C_p} Z_i Z_{k+1-i}\\
& = \frac{q}{1-q} + E\left[ X_n \middle| \left| \partial \qua_n \right| = 2p \right] - \frac{C_{p+1}}{12 C_p} -  \sum_{k=0}^{p-1}  \frac{C_{p-k}Z_{k+1}}{12 C_p}\\
& \quad \underset{p \to \infty}{\longrightarrow}  \frac{q}{1-q} - \frac{1}{2},
\end{align*}
which is negative and bounded away from $0$. This implies that percolation does not occur for $q < 1/3$.

\bigskip

To prove that percolation occurs for $q > 1/3$, we can, in the same way as in Section \ref{sec:siteperc}, compute the law of $U_n$ using the fact that $2X_n = (A_{n+1} - A_n) + (U_{n+1} - U_n)$. This yields, conditionally on $|\partial \qua_n | = 2p $:
\begin{itemize}

\item[(1)] When $y_0, y_1 \notin \partial \qua_n$:
\[U_{n+1} = U_n +2  \stackrel{(\ci)}{-} G_q( U_n +2).\]

\item[(2)] When $y_0 \notin \partial \qua_n$ and $y_1 \in
\partial \qua_n$:
\begin{itemize}
\item if $\qua_n^l$ is infinite,
\begin{equation*}
U_{n+1} =
\begin{cases}
U_n - 2k  \stackrel{(\ci)}{-} G_q(U_n -2k) & \text{ (if $2k < U_n$)},\\
0 & \text{ (if $2k \geq U_n$)},
\end{cases}
\end{equation*}

\item if $\qua_n^r$ is infinite,
\begin{equation*}
U_{n+1} =
\begin{cases}
\big[U_n + 1 \big] \stackrel{(\ci)}{-} G_q(U_n + 1) & \text{
(if $2(p - k)  \geq U_n +2$)},\\
\\
2(p - k) & \text{ w.p. $> 0$ (if $2(p - k) < U_n + 2$)},\\
2(p - k) - 1 \stackrel{(\ci)}{-} G_q(2(p - k) - 1) & \text{ w.p. $> 0$ (if $2(p - k) < U_n + 2$).}
\end{cases}
\end{equation*}
\end{itemize}

\item[(3)]When $y_0 \in \partial \qua_n$ and $y_1 \notin
\partial \qua_n$:
\begin{itemize}
\item if $\qua_n^l$ is infinite,
\begin{equation*}
U_{n+1} =
\begin{cases}
U_n - 2k \stackrel{(\ci)}{-} G_q(U_n -2k) & \text{ (if $2k +1  < U_n$)},\\
1 \stackrel{(\ci)}{-} G_q(1) & \text{ (if $2k +1  \geq U_n$)},
\end{cases}
\end{equation*}

\item if $\qua_n^r$ is infinite,
\begin{equation*}
U_{n+1} =
\begin{cases}
U_n  \stackrel{(\ci)}{-} G_q(U_n) & \text{ (if $2(p - k) - 1
\geq U_n$)},\\
\\
2(p - k) & \text{ w.p. $> 0$ (if $2(p - k) - 1 < U_n$)},\\
2(p -k) - 1 \stackrel{(\ci)}{-} G_q(2(p-k)-1) & \text{ w.p. $> 0$ (if $2(p - k) - 1 < U_n$).}
\end{cases}
\end{equation*}
\end{itemize}

\item[(4)] When $y_0, y_1 \in \partial \qua_n$, we obtain
the same transitions as in case (2) when $\qua_n^l$ is infinite, and
as in case (3) when $\qua_n^r$ is infinite. Finally, if $\qua_n^m$ is
infinite, we write $y_0 = x_{2i}$ ($1 \leq i \leq k$), and
\begin{equation*}
U_{n+1} =
\begin{cases}
0 & \text{ (if $2i > U_n$)},\\
\\
\big[U_n - 2i + 1\big] \\
\stackrel{(\ci)}{-} G_q(U_n - 2i +1) &
\text{ (if $2i \leq U_n \leq 2(p-k+i) -2$)},\\
\\
2(p - k) & \text{ w.p. $> 0$ (if $U_n > 2(p-k+i) -2$)},\\
2(p-k) - 1 \stackrel{(\ci)}{-} G_q(2(p-k)-1) & \text{ w.p. $> 0$ (if $U_n > 2(p-k+i) -2$)}.
\end{cases}
\end{equation*}
\end{itemize}

From here, computations are the same as with $A_n$, except that we have to take extra care of the truncated geometric random variables.
We consider the process $(U'_n)$ coupled with $(U_n)$, and with increments given conditionally on $|\partial \qua_n | = 2p$ by
\begin{equation*}
U'_{n+1} =
\begin{cases}
U'_n + 2 \stackrel{(\ci)}{-} G_q(U_n + 2) & \text{ w. p.
$\frac{C_{p+1}}{12 C_p}$},\\
U'_n - 2k \stackrel{(\ci)}{-} G_q(U_n - 2k) & \text{ w. p.
$2 \frac{C_{p-k}Z_{k+1}}{12 C_p} +  \sum_{i = 1}^ k \frac{C_{p-k}}{12C_p} Z_i Z_{k+1-i}$},\\
U'_n + 1 \stackrel{(\ci)}{-} G_q(U_n + 1) & \text{ w. p.
$ \frac{C_{p-k}Z_{k+1}}{12 C_p}$},\\
U'_n \stackrel{(\ci)}{-} G_q(U_n) & \text{ w. p.
$ \frac{C_{p-k}Z_{k+1}}{12 C_p} +  \sum_{i = 1}^ k \frac{C_{p-k}}{12C_p} Z_i Z_{k+1-i}$},\\
U'_n - 2i + 1 \stackrel{(\ci)}{-} G_q(U_n - 2i +1) & \text{ w. p.
$ \frac{C_{p-k}}{12C_p} Z_i Z_{k+1-i}$ for $1 \leq i \leq k$}.
\end{cases}
\end{equation*}
Let us fix $q>1/3$, and $\varepsilon > 0$. We can choose $C>0$ such that $E\left[G_q(C) \right] > \frac{q}{1-q} - \varepsilon$ and $P(X_n > C | | \partial \qua_n | = 2p ) < \varepsilon$. On the event $\{ U_n > 2C \}$, we have
\begin{align*}
E & \left[U'_{n+1} - U'_n \middle| \left| \partial \qua_n \right| = 2p \right]\\
& \leq 2 \frac{C_{p+1}}{12 C_p} - \sum_{k = 0}^{p-1} k \left( 4  \frac{C_{p-k}Z_{k+1}}{12 C_p} + 3 \sum_{i = 1}^ k \frac{C_{p-k}}{12C_p} Z_i Z_{k+1-i} \right) \\
& \quad + \sum_{k = 1}^{p-1} \sum_{i = 1}^ k (k + 1 - 2i) \frac{C_{p-k}}{12C_p} Z_i Z_{k+1-i} + \sum_{k=0}^{p-1} \frac{C_{p-k}Z_{k+1}}{12 C_p} \\
& \quad - \left( \frac{q}{1-q} - \varepsilon \right) \left( \frac{C_{p+1}}{12 C_p} + \sum_{k=0}^{C-1} \left( 4  \frac{C_{p-k}Z_{k+1}}{12 C_p} +  3 \sum_{i = 1}^ k \frac{C_{p-k}}{12C_p} Z_i Z_{k+1-i} \right) \right)\\
& \leq \frac{C_{p+1}}{12 C_p} + E\left[ X_n \middle| | \partial \qua_n | = 2p \right] +  \sum_{k=0}^{p-1} \frac{C_{p-k}Z_{k+1}}{12 C_p}
- \left( \frac{q}{1-q} - \varepsilon \right) \left( 1 - \varepsilon \right)\\
& \underset{p \to \infty}{\longrightarrow} \frac{1}{2} - \left( \frac{q}{1-q} - \varepsilon \right) \left( 1 - \varepsilon \right).
\end{align*}
This shows that for $q > 1/3$, one has $U_n = O(\ln n)$. Hence, $A_n \to \infty$, so that percolation occurs.

\section*{Acknowledgements}

Part of this work was done while P.N. was affiliated with the Courant Institute (New York University), when it was supported in part by the NSF grants OISE-0730136 and DMS-1007626. L.M. would also like to thank the ForschungsInstitut f\"ur Mathematik at ETH for its hospitality.

\end{document}